\documentclass[11pt,leqno,a4paper,english]{amsart}
\usepackage{amssymb,amsmath,amsfonts,amsthm,amscd}
\usepackage{leftidx}
\usepackage{tikz}
\usepackage[latin1]{inputenc}
 \def\pasdegrille{\let\grille =
\pasgrille}  \def\aat#1#2#3{ \divide
\dimen1 by 48 \dimen3=\dimen1 \multiply \dimen1 by #1 \advance \dimen1
by -\dimen3 \divide \dimen1 by 101 \multiply \dimen1 by 100 \divide
\dimen2 by \count11 \multiply \dimen2 by #2
\setbox0=\hbox{#3}\ht0=0pt\dp0=0pt \rlap{\kern\dimen1 \vbox
to0pt{\kern-\dimen2\box0\vss}}\dimen1= \wd1 \dimen2=\ht1}
\def\pasgrille{ \count12= \dimen1 \divide \count12 by 50 \divide
\dimen2 by \count12 \count11 =\dimen2 \ \divide \dimen1 by 48==
\setlength{\unitlength}{\dimen1} \smash{\rlap{\ }} \dimen1= \wd1
\dimen2=\ht1 } \def\grille{ \count12= \dimen1 \divide \count12 by 50
\divide \dimen2 by \count12 \count11 =\dimen2 \ \divide \dimen1 by 48
\setlength{\unitlength}{\dimen1} \smash{\rlap{\graphpaper[1](0,0)(50,
\count11)}} \dimen1= \wd1 \dimen2=\ht1 }

\usepackage{mathrsfs}

\usepackage{hyperref}
\hypersetup{
    colorlinks=true,       
    linkcolor=blue,          
    citecolor=blue,        
    filecolor=magenta,      
    urlcolor=cyan           
}

\usepackage{subfigure,graphicx}
\usepackage{xcolor}
\usepackage{xspace}

\usepackage{enumerate}

\usepackage[all]{hypcap}

\newcommand{\mb}[1]{\ensuremath{\mathbb{#1}}}
\newcommand{\N}{{\mb{N}}}

\newcommand{\R}{{\mb{R}}}

\newcommand{\eps}{\varepsilon}

\newcommand{\A}{\ensuremath{A}}
\newcommand{\Atlas}{\ensuremath{\mathcal A}}
\newcommand{\D}{\ensuremath{\mathscr D}}
\newcommand{\E}{\ensuremath{\mathcal E}}

\renewcommand{\L}{\ensuremath{\mathcal L}}
\newcommand{\M}{\ensuremath{\mathcal M}}
\renewcommand{\O}{\ensuremath{\mathcal O}}

\renewcommand{\k}{\ensuremath{\kappa}}
\newcommand{\y}{\ensuremath{\varrho}}

\newcommand{\Con}{\ensuremath{\mathscr C}}

\newcommand{\Cinf}{\ensuremath{\mathscr C^\infty}}
\newcommand{\Cinfc}{\ensuremath{\mathscr C^\infty_{c}}}

\newcommand{\X}{\mathcal X} 
\newcommand{\Y}{\mathcal Y} 

\renewcommand{\d}{\ensuremath{\partial}}

\newcommand{\rhs}{r.h.s.\@\xspace}

\newcommand{\resp}{resp.\@\xspace}

\newcommand{\nhd}{neighborhood\xspace}

\newcommand{\suff}{sufficiently\xspace}

\newcommand{\Slim}{S}

\DeclareMathOperator{\Char}{Char}

\newcommand{\tchi}{\tilde{\chi}}
\newcommand{\hchi}{\hat{\chi}}

\newcommand{\tO}{\tilde{\O}}

\newcommand{\tB}{\tilde{B}}

\newcommand{\tk}{\tilde{\kappa}}
\newcommand{\tg}{\tilde{g}}

\newcommand{\ty}{\tilde{y}}
\newcommand{\tmu}{\tilde{\mu}}

\newcommand{\bichar}{bicharacteristic\xspace}
\newcommand{\bichars}{bicharacteristics\xspace}

\newcommand{\bld}[1]{\mbox{\boldmath $#1$}}

\newcommand{\transp}{\ensuremath{\phantom{}^{t}}}

\let \div \relax
\DeclareMathOperator{\div}{div}
\newcommand{\nablag}{\nabla_{\!\! g}}
\newcommand{\divg}{\div_{\! g}}

\DeclareMathOperator{\dist}{dist}

\DeclareMathOperator{\Op}{Op}

\DeclareMathOperator{\supp}{supp}

\DeclareMathOperator{\Hamiltonian}{H}
\newcommand{\Hp}{\Hamiltonian_p}
\newcommand{\Hpref}{\Hamiltonian_{p^0}}
\newcommand{\inp}[2]{(#1, #2)} 
\newcommand{\biginp}[2]{\big(#1, #2 \big)}

\newcommand{\bigdup}[2]{\big\langle #1, #2 \big\rangle}
\newcommand{\dup}[2]{\langle #1, #2 \rangle}

\newcommand{\ovl}[1]{\overline{#1}}

\newtheoremstyle{note}{} {}{\itshape}{-6pt}{\bf}{. --}{ }{}

\newtheorem{theorem}{Theorem}[section]
\newtheorem{proposition}[theorem]{Proposition}
\newtheorem{lemma}[theorem]{Lemma}
\newtheorem{corollary}[theorem]{Corollary}

\newtheorem*{theo*}{Theorem}

\newtheorem{theorembis}{Theorem}
\newtheorem{propositionbis}{Proposition}
\newcounter{theorembiss}

\newcounter{propositionbiss}

\theoremstyle{definition}
\newtheorem{definition}[theorem]{Definition}
\newtheorem{remark}[theorem]{Remark}

\newtheorem{definitionbis}{Definition}
\newcounter{definitionbiss}
\newcounter{sectionbiss}

\newtheorem{definitionter}{Definition}
\newcounter{definitionterr}
\newcounter{sectionterr}

\DeclareMathOperator{\loc}{loc}
\DeclareMathOperator{\comp}{comp}
\DeclareMathOperator{\phg}{ph}

\DeclareMathOperator{\obs}{obs}
\newcommand{\Cobs}{C_{\obs}}

\newcommand{\Norm}[2]{{\| #1 \|}_{#2}}
\newcommand{\bigNorm}[2]{{\big\| #1 \big\|}_{#2}}

\numberwithin{equation}{section}


\title[Measure propagation and wave controllability]
{Measure propagation along a $\Con^0$-vector field and wave controllability on a rough compact manifold} 

\author{Nicolas Burq}
\address{Nicolas Burq. Laboratoire de Math\'ematiques d'Orsay, Universit\'e
  Paris-Saclay, B\^atiment~307, 91405
  Orsay Cedex \& CNRS UMR 8628 \& Institut universitaire de France}
 \email{nicolas.burq@universite-paris-saclay.fr}
 \author{Belhassen Dehman}
  \address{Belhassen Dehman. Universit\'e de Tunis El Manar, Facult\'e
  des Sciences de Tunis, 2092 El  Manar \& Ecole Nationale d'Ing\'enieurs de Tunis, ENIT-LAMSIN, B.P. 37, 1002 Tunis, Tunisia. }
\email{Belhassen.Dehman@fst.utm.tn}
 \author{J\'er\^ome Le Rousseau}
 \address{J\'er\^ome Le Rousseau.
   Universit\'e Sorbonne Paris Nord, Laboratoire Analyse, G\'eom\'etrie et Applications, LAGA, CNRS, UMR 7539, F-93430, Villetaneuse, France.}
\email{jlr@math.univ-paris13.fr}

\date{\today}

\begin{document}
\begin{abstract}
  The celebrated Rauch-Taylor/Bardos-Lebeau-Rauch geometric control condition is central in
  the study of the observability of the wave equation linking this
  property to high-frequency propagation along geodesics that are the
  rays of geometric optics. This connection is best understood
  through the propagation properties of microlocal defect measures
  that appear as solutions to the wave equation concentrate. For a
  \suff smooth metric this propagation occurs along the
  bicharacteristic flow.  If one considers a merely $\Con^1$-metric
  this bicharacteristic flow may however not exist. The Hamiltonian
  vector field is only continuous; bicharacteristics do exist (as integral curves of this continuous vector field) but
  uniqueness is lost.  Here, on a compact manifold without boundary,
  we consider this low regularity setting, revisit the geometric
  control condition, and address the question of support propagation
  for a measure solution to an ODE with continuous coefficients. This
  leads to a sufficient condition for the observability and
  equivalently the exact controllability of the wave equation.
  Moreover, we investigate the stability of the observability
  property and the sensitivity of the control process under
  a perturbation of the metric of regularity  as low as Lipschitz.
\end{abstract} 

\maketitle
\tableofcontents

\section{Introduction}

The observability property for the wave equation has been intensively
studied during the last decades mainly because of its deep connection
with the problem of exact controllability. Until the end of the 80's,
most of the positive results of observability were established under a
(global) geometric assumption, the so-called $\Gamma$-condition
introduced by J.-L.~Lions, essentially based on and well-adapted to a
multiplier method \cite{Lions}. Later, following \cite {Rauch-Taylor},
Bardos, Lebeau and Rauch established in \cite{BLR:92}, boundary
observability inequalities under a geometric control condition (GCC for
short), linking the set on which the control acts and the generalized
geodesic flow. Proofs of this result are based on microlocal tools,
such as the propagation in phase space of wavefront sets in \cite{BLR:92} or 
the propagation of microlocal defect
measures in more modern proofs \cite{BurqGerard}. For the latter
approach, microlocal defect
measures originate from the concentration phenomena for  sequences of waves
 if one assumes that observability does not hold. Away
from boundaries one obtains 
\begin{align}
  \label{eq: ODE intro}
\transp{\Hp} \mu =0,
\end{align}
yielding the transport of the measure $\mu$ along the bicharacteristic
flow in phase space. This flow is generated by the hamiltonian vector field $\Hp$ 
associated with the symbol of the wave operator $p$.
However note that despite their high efficiency and
robustness, these methods present the great disadvantage of requiring
too much regularity in the coefficients of the wave operator and the
geometry. To define the generalized bicharacteristic flow and prove
the propagation properties mentioned above a minimal smoothness of the
metric and the boundary domain is needed. 
To our knowledge, the best
result, in the context of $\Con^2$ metrics, was proven in \cite{Burq}, and barely misses
the natural minimal smoothness required to define the geodesic flow
($W^{2, \infty}$) and thus the geometric control condition.

In this context, in the present article, we address the following
natural question: how can one derive observability estimate for the
wave equation from optimal observation regions in the case of a
nonsmooth metric?  This problem has already received some attention
and answers by E.~Zuazua and his collaborators, in
\cite{CastroZuazua02, CZ07}, and more recently in \cite{Fan-Zua} (see
also the result of \cite{Dehman-Ervedoza}). More precisely, in
\cite{CastroZuazua02, CZ07}, the authors prove a lack of observability
of waves in highly heterogeneous media, that is, if the density is of
low regularity. In \cite{Fan-Zua}, the authors establish observability
with coefficients in the Zygmund class and also observability with
loss when the coefficients are Log-Zygmund or
Log-Lipschitz. Furthermore, this result is proven sharp since one
observes an infinite loss of derivatives for a regularity lower than
Log-Lipschitz.  Note that these analyses are carried out in 
one space dimension. This calls for the following comments. First, in this simplified framework, for smooth coefficients all
the geodesics reach the observability region in uniform time:  captive geodesics
are not an issue. Second, proofs are based on a sidewise energy
estimate, a technique that is 
specific to the one-dimensional setting; the underlying idea consists of exchanging the roles of the
time and space variables and, finally in proving hyperbolic
energy estimates for waves with rough coefficients.  Unfortunately,
such a method does not extend to higher space dimensions. Furthermore, for
the low regularity considered in these articles, the geodesic flow is not well
defined. Proving propagation results for wavefront sets or microlocal
defect measure appears quite out of reach in such cases.

The present work is the first in a series of three articles devoted to
the question of observability (and equivalently exact controllability)
of wave equations with nonsmooth coefficients. Here, we initiate this
study on a compact Riemannian manifold with a rough metric, yet {\em
without boundary}, while the two forthcoming articles will present the
counterpart analysis on manifolds with boundary (or bounded domains of
$\R^d$) \cite{B-D-LR1, B-D-LR2}. The presence of a boundary yields a
much more involved analysis and in \cite{B-D-LR1, B-D-LR2} we develop
Melrose-Sj\"ostrand generalized propagation theory in a low regularity
framework. In the present article, our main result is the
observability of the wave equation with a $\Con^1$-metric, completed
with the stability of the observability property for small Lipschitz
($W^{1, \infty}$) perturbations of the metric.  More precisely, we
first show that if the geometric control condition in time $T$  holds for
geodesics associated with a $\Con^1$-metric $g$, 
then the
observability property holds for the wave equation, and equivalently
exact controllability. For this low regularity case one has to
carefully consider the meaning of the geometric condition (or more generally the meaning of a geodesic) since the
metric does not define a natural geodesic flow: geodesics are not
uniquely defined. Only their existence is guaranteed.
Second, we consider a reference $\Con^1$-metric $g^0$ as above and  we prove 
that observability  also holds for any Lipschitz metric $g$ chosen \suff
close to $g^0$ (in the Lipschitz topology).
It has to be noticed that Lipschitz metrics are too rough to permit the use of microlocal tools
and a direct proof of the observability property.  Even worse for such
a metric the geometric control condition itself does not seem
make sense (as the generating vector field is only $L^\infty$), and we have to use a perturbation argument near the (not so) smooth $\Con^1$ reference metric.

Following the strategy of \cite{Burq},
we argue by contradiction and we prove a 
propagation result for microlocal defect measures in a low
regularity setting. We prove that these measures are solutions to the ODE~\eqref{eq: ODE intro} with
 here $\Hp$ having $\Con^0$-coefficients. Then, we deduce some general properties about their support. Namely we show that their support is a union of integral curves of the vector field. This latter step also follows from Ambrosio and Crippa's superposition principle~\cite
{AmCr14}. Yet,  we give a completely different proof which is of
interest since it can be extended
to the case of a domain with a boundary \cite{B-D-LR1, B-D-LR2}. We
have not been able to extend the approach of \cite{AmCr14} to that
case. To derive the ODE fulfilled by the microlocal defect measure we
heavily rely on some harmonic analysis results due to R.~Coifman and
Y.~Meyer~\cite[Proposition IV.7]{Coifman-Meyer} that express that
the commutator of a pseudo-differential operator of order 1 and a
Lipschitz function is a bounded operator on $L^2$.

Finally, going further in the analysis, we investigate another stability property with respect to perturbations of the metric. We prove that the HUM optimal control associated with a fixed initial
data is {\em not} stable with respect to perturbations of the metric.

\bigskip

\noindent\textbf{Acknowledgements.}
The research of N. Burq was partially supported by Agence Nationale de la
Recherche through projects ANA\'E ANR-13-BS01-0010-03, ISDEEC
ANR-16-CE40-0013, Institut universitaire de France, and the European
Research Council (ERC) under the European Union's Horizon 2020
research and innovation programme (Grant agreement 101097172 -
GEOEDP).  The research of B. Dehman was partially supported by  the Tunisian Ministry for Higher Education and
Scientific Research within the LR-99-ES20 program. Finally, the
authors wish to thank the anonymous reviewer for pointing out
arguments that needed clarifications.

\subsection{Outline}
The article is organized as follows.  In
Section~\ref{Geometric setting} we set up the geometric framework
we shall use and in Section~\ref{Exact controllability and
  observability} we precisely recall the equivalence of observability
and exact controllability for the wave equation.  In Section
\ref{Statement of the results} we state the main results of the
article.

In Section \ref{sec: Geometric facts} we recall some
geometric facts and the notions of pseudo-differential calculus and
microlocal defect (density) measures on a manifold. In addition, using
bicharacteristics we state the geometric control condition of
\cite{BLR:92} in its classical form ($\Con^2$-metric) and generalized
form ($\Con^1$-metric).

In Section~\ref{sec: mdm and pde} we recall what microlocal defect
measures are and we show how, if associated with sequences of solutions
of PDEs, their support can be estimated and how a transport ODE can be
derived, in the particular context of low regularity of coefficients.

Section~\ref{sec proof: th: ODE} is devoted to our proof of the
support propagation for measures solutions of a ODE with
$\Con^0$-coefficients, Theorem~\ref{th: ODE}.

In Section \ref{prooftheoprinc} we use the results of
Section~\ref{sec: mdm and pde} and the propagation result of
Theorem~\ref{th: ODE} to prove the observability and controllability
results for the wave equation, Theorems \ref{theo2} and \ref{theoprinc}.

Finally, in Section~\ref{sec: lack continuity HUM} we prove the
results related to stability properties of the HUM control process.

\subsection{Setting and well-posedness}
\label{Geometric setting}

Throughout the article, we consider $\M$ a $d$-dimensional $\Cinf$-compact manifold, that
is, a manifold without boundary with a topology that makes it
compact equipped with a $\Cinf$-atlas. We assume that the topology is also given by a
Riemannian metric $g$, to be chosen either Lipschitz or of class $\Con^k$ for
some value of $k$ to be made precise below\footnote{Note that despite
  considering $\Con^k$ metrics with $k < \infty$ we still impose the
  underlying manifold to be smooth. This is due to our use of
  pseudo-differential techniques that are simple to introduce on a
  smooth manifold. See Section~\ref{sec: measures symbols operators}}. 

We denote by $\mu_g$ the canonical positive Riemannian density on
$\M$, that is, the density measure associated with the density
function $(\det g)^{1/2}$. We also consider a positive Lipschitz or of class 
 $\Con^{k}$-function
$\k$ and we define the density $ \k \mu_g$.

The $L^2$-inner product and norm are considered with respect to this
density $\k \mu_g$, that is, 
\begin{align}
  \label{eq: L2-inner product and norm}
  \inp{u}{v}_{L^2(\M)} = \int_\M u \bar{v} \, \k \mu_g, 
  \qquad 
  \Norm{u}{L^2(\M)}^2 = \int_\M |u|^2\, \k \mu_g.
\end{align}
We denote by $L^2 V(\M)$ the space of $L^2$-vector fields on $\M$,
equipped with the norm
\begin{align*}
  \Norm{v}{L^2 V(\M)}^2
  = \int_\M  g(v, \bar{v}) \, \k \mu_g, \qquad v \in L^2 V(\M).
\end{align*}
We recall that the Riemannian gradient and divergence are given by 
\begin{align*}
  g (\nablag f, v) = v(f)
  \ \ \text{and} \ \ 
  \int_\M f \divg v \mu_g = - \int_\M v(f) \, \mu_g,
\end{align*}
for $f$ a  function and $v$ a vector field, yielding in local
coordinates
\begin{align*}
  (\nablag f)^i= \sum_{1\leq j \leq d} g^{i j}\d_{x_j} f, 
  \qquad 
  \divg v = (\det g)^{-1/2} 
  \sum_{1\leq i\leq d} \d_{x_i}\big( (\det g)^{1/2} v^i \big),
\end{align*}
with $(g_x^{ij}) = (g_{x,i j})^{-1}$.

We introduce the elliptic operator $\A = \A_{\k,g}= \k^{-1} \divg (\k
\nablag)$, that is, in local coordinates
\begin{align*}
  \A f = \k^{-1} (\det g)^{-1/2} 
  \sum_{1\leq i, j \leq d}  \d_{x_i}\big( 
  \k  (\det g)^{1/2} g^{i j}(x)\d_{x_j} f
  \big).
\end{align*}
Its principal symbol is simply
$a(x,\xi) = - \sum_{1\leq i, j \leq d} g_x^{i j} \xi_i \xi_j$. Note that
for $\k=1$, one has $\A =  \Delta_g$, the Laplace-Beltrami operator
associated with $g$ on $\M$.  Similarly to $\Delta_g$, the operator
$\A$ is unbounded on $L^2(\M)$. With the domain $D(\A) =H^2(\M)$ one
finds that $\A$ is self-adjoint, with respect to the $L^2$-inner product
given in \eqref{eq: L2-inner product and norm}, and
negative. Moreover, one has
\begin{align*}
  \inp{\A u}{v}_{L^2(\M)} 
  = - \int_\M g (\nablag u, \nablag \bar{v}) \,\k \mu_g, 
  \qquad u\in H^2(\M), \ v \in H^1(\M).
\end{align*}
Together with $\A$ we consider the wave operator $ P_{\k,g}
= \d_t^2 - \A_{\k,g}  +m$, with $m >0$ a constant
and the equation
\begin{align}
  \label{eq: wave equation-intro}
  \begin{cases}
     P_{\k,g} y =f 
    & \text{in}\ (0,+\infty)\times\M,\\
      y_{|t=0}  = y^0, \ \d_t y_{|t=0}  = y^1 
      & \text{in}\ \M.
  \end{cases}
\end{align}
It is well-posed in the energy space $H^1(\M) \oplus L^2(\M)$.
\begin{proposition}
  \label{prop: well-posedness wave equation}
  Consider $\k$ and $g$ both of Lipschitz class. 
  Let $(y^0, y^1) \in H^1(\M) \times L^2(\M)$ and let $f \in
  L^2\big(0,T; L^2(\M)\big)$, for any $T>0$.   There exists a unique  
  \begin{align*}
    y \in \Con^0 \big([0,+\infty); H^1(\M)  \big)
    \cap \Con^1 \big([0,+\infty); L^2(\M)\big)
  \end{align*}
  that is a weak solution of \eqref{eq: wave equation-intro}, that is, 
  $ y_{|t=0}  = y^0$ and $\d_t y_{|t=0}  = y^1$ and
  \begin{align*}
    P_{\k,g}  y =f \quad \text{in} \
    \D' \big((0,+\infty) \times \M\big).
  \end{align*}
\end{proposition}
\begin{remark}
  At this level of regularity of $\k$ and $g$, the well-posedness of
  the wave equation is classical. For less regular
  coefficients we refer to \cite{Colombini-Del Santo} and
  \cite{Colombini-DelSanto-Fanelli-Metivier}.
\end{remark}
In what follows, for simplicity we shall consider the case $m=1$, that
is for 
\begin{align*}
  P_{\k,g} = \d_t^2 - \A_{\k,g}  +1.
\end{align*}
In this case, we denote by 
\begin{align*}
  \E_{\k, g} (y) (t)
  &= \frac12 \big( \Norm{y(t)}{H^1(\M)}^2
    + \Norm{\d_ty(t)}{L^2(\M)}^2 \big)\\
  &= \frac12 \big( \Norm{y(t)}{L^2(\M)}^2
    + \Norm{\nablag y(t)}{L^2 V(\M)}^2
    + \Norm{\d_ty(t)}{L^2(\M)}^2
    \big),
\end{align*}
the energy of this solution at time $t$.
For a weak solution $y$ of \eqref{eq: wave equation-intro}, if $f=0$ this energy
is independent of time $t$, that is, 
\begin{align*}
  \E_{\k, g} (y) (t) = \E_{\k, g} (y) (0) = \frac12 \big( \Norm{y^0}{H^1(\M)}^2
  + \Norm{y^1}{L^2(\M)}^2 \big).
\end{align*}

\begin{remark}
  The equation we consider, with the constant $m>0$, is often referred
  to the Klein-Gordon equation. Here, we keep the name wave
  equation. We choose this equation instead of the classical wave
  equation that corresponds to the case $m=0$.  In fact, on a compact
  manifold without boundary, constants are eigenfunctions of the
  elliptic operator $\A_{\k,g}$ with $0$ as an eigenvalue. Hence,
  constant functions are solutions to the wave equation and are
  so-called {\em invisible} solutions, as far as the observability
  property we are interested in is concerned. If one
  considers a manifold with boundary and say, homogeneous Dirichlet
  conditions, this issue becomes irrelevant. We could have dealt with the case $m=0$ (the usual wave equation) at the cost of additional technical complications.
\end{remark}

\subsection{Exact controllability and observability}
\label{Exact controllability and observability}

Let $\omega$ be a nonempty open subset of $\M$ and $T>0$.  The notion of exact
controllability for the wave equation from $\omega$ at time
$T$ is stated as follows.
\begin{definition}[exact controllability in $H^1(\M) \oplus L^2(\M)$]
  \label{def: exact controllability}
  
  One says that the wave equation is exactly controllable from
  $\omega$ at time $T>0$ if
  for any
  $(y^0, y^1) \in H^1(\M) \times L^2(\M)$, there exists  
  $f\in L^2((0,T)\times \M)$ such that the weak solution $y$ to
  \begin{equation}
    \label{eq.controlint}
    P_{\k,g} y =  \bld{1}_{(0,T) \times \omega} \, f, \quad (y_{|t=0}, \d_t y_{|t=0}) = (y^0, y^1),
  \end{equation}
  as given by Proposition~\ref{prop: well-posedness wave equation}
  satisfies $(y, \d_t y)_{|t= T} = (0,0)$.
 The function $f$ is called the control function or simply the
 control. 
\end{definition}
Observability of the wave equation from the open set $\omega$
in time $T$ is the following notion.
\begin{definition}[observability]
  \label{def: observability}
  One says that the wave equation is observable from
  $\omega$ at time $T$ if there exists $\Cobs>0$ such that 
  for any
  $(u^0, u^1) \in H^1(\M) \times L^2(\M)$ one has 
  \begin{equation}
    \label{eq: observability}
    \E_{\k, g} (u) (0)
    \leq \Cobs
    \Norm{\bld{1}_{(0,T) \times \omega}\, \d_t u}{L^2(\L)}^2,
  \end{equation}
  for $u \in \Con^0 \big([0,T]; H^1(\M)  \big)
    \cap \Con^1 \big([0,T]; L^2(\M)\big)$ the weak solution of $P_{\k,g}  u=0$
    with $ u_{|t=0}  = u^0$ and $\d_t u_{|t=0}  = u^1$ as given by Proposition~\ref{prop: well-posedness wave equation}; see~\cite{Lions}.
  \end{definition}
\begin{proposition}
  \label{prop: equiv controllability observability}
  Let $\omega$ be an open subset of $\M$ and $T>0$.
  The wave equation is exactly controllable from
  $\omega$ at time $T$ if and only if it is observable from
  $\omega$ at time $T$.
\end{proposition}
\begin{remark}
  In the case $m=0$ the
  energy function is given by
  \begin{align*}
\E_{\k, g} (u) (t)= \frac12 \big( \Norm{\d_tu(t)}{L^2(\M)}^2
  + \Norm{\nablag u(t)}{L^2 V(\M)}^2
  \big).
  \end{align*}
  It follows that a constant function $u$, a solution to the wave
  equation $(\d_t^2 - \A)u=0$ has zero energy. Since
  $\Norm{\bld{1}_{(0,T) \times \omega}\, \d_t u}{L^2(\L)}^2$ also
  vanishes, one sees that such a solution are invisible for an
  observability inequality of the form of \eqref{eq:
    observability}. Possibilities to overcome this difficulty are to work
  in a quotient space or to change the wave operator into the
  Klein-Gordon operator. Here, we chose for simplicity the latter option.
\end{remark}

\subsection{Main results} 
\label{Statement of the results}

We introduce the following spaces for the coefficients $(\k,g)$ to
distinguish various levels of regularity:
\begin{align*}
  &\X^2(\M)= \{ (\k,g);\ \k\in  \Con^2(\M) \ \text{and} \  \ g \ \text{is a}\
    \Con^2\text{-metric on}\ \M \}, \\
  &\X^1(\M)= \{ (\k,g);\ \k\in  \Con^1(\M)  \ \text{and} \  \ g \ \text{is a}\
    \Con^1\text{-metric on}\ \M \}, \\
  &\Y(\M)= \{ (\k,g);\ \k\in  W^{1, \infty} (\M)  \ \text{and} \ \ g \ \text{is a}\
    W^{1, \infty}\text{-metric on}\ \M \}.
\end{align*}
We start by  recalling the controllability result known for
regularity higher than or equal to $\Con^2$, under the Rauch-Taylor
geometric control condition.
\begin{definition}[Rauch-Taylor, geometric control condition]
  \label{def: Geometric control condition-thm}
  Let $g$ be a $\Con^k$ metric, $k=1$ or $2$, and let $\omega$ be an open set of $\M$ and
  $T>0$. One says that $(\omega, T)$ fulfills the geometric control
  condition if all maximal geodesics associated with $g$, traveled at
  speed 1, encounter $\omega$ for some time $t \in (0,T)$.
\end{definition}
\setcounter{sectionbiss}{\arabic{section}}
\setcounter{definitionbiss}{\arabic{theorem}}
\setcounter{sectionterr}{\arabic{section}}
\setcounter{definitionterr}{\arabic{theorem}}
A second formulation of this geometric  condition based on the dual notion of
\bichars is given in Section~\ref{sec: Geometric control condition}
below.

\begin{theorem}[Exact controllability  --
  $\Con^2$-regularity]
  \label{theorem: Burq 97}
  Consider $(\k, g) \in \X^2(\M)$, $\omega$ an open subset of $\M$ and
  $T>0$ such that $(\omega, T)$ fulfills the geometric control
  condition of Definition~\ref{def: Geometric control
    condition-thm}. Then, the wave equation is exactly
  controllable from $\omega$ at time $T$.
\end{theorem}
This result was first proven by Rauch and Taylor \cite{Rauch-Taylor}
for a smooth metric. The case $(\k, g) \in \X^2(\M)$ was proven by
the first author in \cite{Burq}. On smooth open sets of $\R^d$, or
equivalently on manifolds with boundary equipped with smooth $(\k, g)$,
for instance in the case of homogeneous Dirichlet boundary conditions, this result
is given in the celebrated articles of Bardos, Lebeau and Rauch \cite{BLR:87,BLR:92}.

\medskip In the present article, we extend the result of
Theorem~\ref{theorem: Burq 97} to cases of rougher coefficients. Our extension
is twofold: we treat the case $(\k,g) \in \X^1(\M)$ and, we treat small
perturbations in $\Y(\M)$ of some $(\k,g) \in \X^1(\M)$. Most importantly, these two
results rely on the understanding of the structure of the support of a
nonnegative measure subject to a homogeneous transport equation with
continuous coefficients. 

\subsubsection{Transport equation and measure support}

Let $\O$ be an open set of a smooth manifold. 
We denote by $^1\D'(\O)$ and $^1\D^{\prime,0}(\O)$ the spaces
of density  distributions and density Radon measures on
$\O$. 

Consider a continuous vector field $X$ on $\O$ and let $\mu$ be a
nonnegative measure density on~$\O$. Assume that $\mu$ is such
that $\transp{X} \mu=0$ in the sense of distributions, that is,
\begin{equation}
  \label{EDO}  
  \dup{\transp{X} \mu}{a}_{^1\D'(\O), \Cinf_c(\O)}
  = \dup{\mu}{X a}_{^1\D^{\prime,0}(\O), \Con^0_c(\O)}=0, 
  \qquad a \in \Cinfc(\O).
\end{equation}
If $X$ is moreover Lipschitz, one concludes that $\mu$ is invariant
along the flow that $X$ generates. However, if $X$ is not Lipschitz,
there is no such flow in general. Yet, integral curves do exist by the
Cauchy-Peano theorem. The following theorem provides a structure of the
support of $\mu$. 
\begin{theorem}
\label{th: ODE}
Let $X$ be a continuous vector field on $\O$ and $\mu$ be a
nonnegative density measure on $\O$ that is a solution to $\transp{X} \mu=0$
in the sense of distributions.  Then, the support of $\mu$
is a union of maximally extended integral curves of the vector field
$X$.
\end{theorem}
In other words, if $m^0\in \O$ is in  $\supp (\mu)$, then there exist an
 interval $I$ in $\R$ with  $0\in I$  and a $\Con^1$ curve $\gamma: I
 \to  \O$ that cannot be extended such that $\gamma(0) = m^0$ and 
\begin{equation*}
  \frac{d }{ds} \gamma(s) = X( \gamma (s)), \qquad s \in I,
\end{equation*} 
and 
$\gamma(I) \subset \supp(\mu)$.

Theorem~\ref{th: ODE} can actually be obtained as a consequence of the
superposition principle of L.~Ambrosio and G.~Crippa
\cite[Theorem~3.4]{AmCr14}. Here, we provide an alternative proof that
is of interest as it allows one to extend this measure support
structure result to the case of an open set or a manifold with
boundary \cite{B-D-LR2} as needed for our application to observability
and controllability. Ambrosio and Crippa's proof is based on a
smoothing-by-convolution argument. Extending this approach does not
seem to be straightforward in the context of a boundary.

Theorem~\ref{th: ODE} is proven in Section~\ref{sec proof: th: ODE}
and its proof is independent of the other sections of the article. A
reader only interested in our proof of Theorem~\ref{th: ODE} may thus
head to Section~\ref{sec proof: th: ODE} directly.

\subsubsection{Exact controllability results}

If $(\k,g) \in \X^2(\M)$, $x \in \M$ and $v \in T_x\M$ there is a unique
geodesic originating from $x$ in direction $v$. In the case
$(\k,g) \in \X^1(\M)$ uniqueness is {\em lost}.  Existence holds however
and maximal (here global, see below) geodesics can still be defined by the Cauchy-Peano theorem. In particular, the geometric
control condition of Definition~\ref{def: Geometric control
  condition-thm} still makes sense. As announced above, our first
result is the following theorem.
\begin{theorem}[Exact controllability  --  $\Con^1$-regularity]
  \label{theo2}
  Consider $ (\k, g) \in \X^1(\M)$, $\omega$ an open subset of $\M$ and
  $T>0$ such that $(\omega, T)$ fulfills the geometric control
  condition of Definition~\ref{def: Geometric control
    condition-thm}.  Then, the wave equation is exactly controllable
  from $\omega$ at time $T$.
\end{theorem}
A second result is the following {\em perturbation} result.
\begin{theorem}[Exact controllability  -- Lipschitz perturbation]
  \label{theoprinc}
  Let $(\k^0, g^0) \in \X^1(\M)$, $\omega$ be an open subset of $\M$ and
  $T>0$ be such that $(\omega, T)$ fulfills the geometric control
  condition of Definition~\ref{def: Geometric control
    condition-thm} with respect to the metric $g^0$. There exists $\varepsilon >0$ such that
  for any $(\k, g) \in \Y(\M)$ satisfying
  \begin{align*}
    \Norm{(\k, g)- (\k^0, g^0)}{\Y(\M)} \leq \varepsilon ,
  \end{align*}
  the wave equation associated with $(\k, g)$  is exactly controllable  by $\omega$ in time $T$.
\end{theorem}
Observe that Theorem~\ref{theo2} is a direct consequence of
Theorem~\ref{theoprinc}. We shall thus concentrate on this second more
general result. Its proof relies on the measure support structure
result of Theorem~\ref{th: ODE}.

The sequence of Theorems~\ref{theorem: Burq 97}, \ref{theo2}, and \ref{theoprinc}
calls for the following important comment. Under the assumption of
Theorem~\ref{theorem: Burq 97}, that is, $ (\k, g) \in \X^2(\M)$, there is
a {\em geodesic flow} and the geometric condition of
Definition~\ref{def: Geometric control condition-thm} is actually a
condition on the flow.  Under the assumption of Theorem~\ref{theo2},
that is, $ (\k, g) \in \X^1(\M)$, as pointed out above there is no
geodesic flow in general. Yet, maximal geodesics are still well
defined and, the geometric condition of Definition~\ref{def:
  Geometric control condition-thm} makes sense because it does not
refer to a flow. However, under the assumption of
Theorem~\ref{theoprinc}, that is, $(\k, g) \in \Y(\M)$, geodesics cannot
be defined in general. No geometric condition can be formulated. Yet,
Theorem~\ref{theoprinc} is a perturbation result and a geometric
condition is expressed for a reference pair $(\k^0, g^0) \in \X^1(\M)$
around which a (small) \nhd in $\Y(\M)$ is considered.

  The following
remark further emphasizes that the perturbation is to be considered around a pair
$(\k^0, g^0) \in \X^1(\M)$ for which the geometric control condition holds
and not around a pair $(\k^0, g^0) \in \X^1(\M)$ for which exact
controllability (or equivalently observability) holds.
\begin{remark}[On the perturbation result] Having both our results, geometric control for $\Con^1$ metrics and Lipschitz stability of exact controllability around a reference metric satisfying the geometric control condition, a natural question is whether the exact controllability property is itself stable by perturbation.    On the one hand, it is classical that  the exact controllability property  is {\em stable}
  under lower-order perturbations of the elliptic operator
  $\A_{\k,g}$ but on the other hand, it is possible to show that it is {\em not} stable under (smooth)
  perturbations of the geometry or the metric. 
  
  Let us illustrate this instability property with a quite simple example. Consider the wave equation on the sphere
$$\mathbb{S}^{d}= \{ x\in \mathbb{R}^{d+1}; \sum_{i} x_i ^2 =1\},$$
endowed with its standard metric and with control domain the open
hemisphere
$$ \omega = \{ x \in \mathbb{S}^d; x_1 >0\}.$$
Even though $\omega$ does not fulfill the geometric control condition
of Definition~\ref{def: Geometric control condition-thm} exact controllability holds for this geometry by an unpublished result  by
G.~Lebeau (see \cite[Section VI.B]{Leb} and \cite{Zhu} for extensions). Consider now the sphere
endowed with the above standard metric, with the smaller control domain
$$\omega_\eps=\{ x \in \mathbb{S}^d; x_1 >\eps\},$$ for some $\eps>0$. This
second geometry is $\eps$-close to the Lebeau example in the
$\Cinf$-topology. Yet, for all $\varepsilon >0$, exact controllability does {\em not}
hold, because there exists a geodesic (the equator, $\{ x\in \mathbb{S}^d; x_1=0\}$) that does not encounter $\overline{\omega_\varepsilon}$.  This shows that in Theorem~\ref{theoprinc}, the assumption that
the reference geometry should satisfy the geometric control condition
{\em cannot} be replaced by the weaker assumption that  it should satisfy the exact
controllability property. This also shows that our perturbation
argument will have to be performed on the actual proof that
geometric control implies exact controllability and {\em not} on the
final property itself.
\end{remark}

\subsubsection{Further results on the control operator}
We finish this section with results analyzing the influence of some
metric perturbations on the control process.

We introduce further levels of regularity for the coefficients by
setting for $k \in \N\cup \{ +\infty\}$, 
\begin{align*}
  \X^k(\M)= \{ (\k,g);\ \k\in  \Con^k(\M) \ \text{and} \  \ g \ \text{is a}\
    \Con^k\text{-metric on}\ \M \}.
\end{align*}
First, we consider $k\geq 2$. We recall the notation
$P_{\k, g} = \d_t^2 - \A_{\k, g}+1$ with
$\A_{\k,g}= \k^{-1} \divg (\k \nablag)$, and we assume that
$ (\k, g)\in \X^k(\M)$, and that $(\omega, T)$ satisfies the geometric
control condition of Definition~\ref{def: Geometric control
  condition-thm} for geodesics given by the metric $g$. Then, by
Theorem~\ref{theorem: Burq 97}, given
$(y^0,y^1) \in H^1(\M) \times L^2(\M)$, there exists
$f\in L^2((0,T)\times \omega)$ such that the solution to
\eqref{eq.controlint} satisfies $y(T) =0$ and $\d_t y(T) =0$.  One can
prove that among all possible control functions there is one of
minimal $L^2$-norm. We denote by $f_{\k, g}^{y^0,y^1}$ this control
function usually named the HUM control function; see for instance
\cite{Lions}. Moreover, the map
\begin{align}
  \label{eq: def HUM operator}
  H_{\k, g}: H^1(\M) \oplus L^2(\M) &\to L^2((0,T)\times \M)\\
                      (y^0,y^1) &\mapsto      f_{\k, g}^{y^0,y^1},\notag
\end{align}
is continuous. Note that $f_{\k, g}^{y^0,y^1}$ is actually a weak solution of the
wave equation with initial data in $L^2(\M)\times H^{-1}(\M)$, meaning
that one moreover has $f_{\k, g}^{y^0,y^1} \in \Con^0([0,T],L^2(\M))$. 
\begin{theorem}[Lack of continuity of the HUM-operator -- Case $k\geq 2$]
  \label{HUM operator}
  Let $k \geq 2$ and $(\k, g)$ as above.
   For any neighborhood $\mathcal{U}$ of $(\k, g)$ in
   $\X^k(\M)$, there exist $ (\tk, \tg) \in \mathcal{U}$ and an initial
   data $(y^{0},y^{1}) \in H^1(\M) \times  L^2(\M)$,
   with $ \Norm{y^0}{H^1}^2+\Norm{y^1}{L^2}^2 =1$, such that the
    respective solutions $y$ and $\ty$ of
\begin{equation}
  \label{eq: stab-HUM}
  \begin{cases}
    P_{\k, g} y 
    =  \bld{1}_{(0,T) \times \omega} \,f ^{y^0,y^1}_{\k,  g} & \text{in} \  (0,T)\times \M,\\
    (y,\d _{t}y)_{|t=0}=(y^0,y^1) & \text{in} \ \M,
  \end{cases}
  \qquad 
  \begin{cases}
    P_{\tk, \tg} \ty 
    =  \bld{1}_{(0,T) \times \omega} \,f ^{y^0,y^1}_{\k,  g} & \text{in} \  (0,T)\times \M,\\
    (\ty ,\d _{t}\ty )_{|t=0}=(y^0,y^1) & \text{in} \ \M,
  \end{cases}
\end{equation}
are such that 
\begin{equation}\label{behaviorHUM}
\E_{\k, g}(\ty - y)(T) = \E_{\k, g}(\ty)(T) \geq 1/2 .
\end{equation}
Moreover, there exists $C_T > 0$ such that
\begin{equation}\label{HUM-noncontinu}
 \bigNorm{(H_{\k, g} - H_{\tk,\tg})(y^0,y^1)}{L^2((0,T)\times \omega)} 
 = \Norm{f^{y^0,y^1}_{\k, g} - f ^{y^0,y^1}_{\tk, \tg}}{L^2((0,T)\times \omega)} \geq C_T,
\end{equation}
for $(y^0, y^1)$ as given above. 
\end{theorem}
\setcounter{theorembiss}{\arabic{theorem}}
 \begin{remark}
 The result of Theorem~\ref{HUM operator} states that starting from the same initial data and solving the two wave equations with the same control vector $f_{\k, g}$ associated with $P_{\k, g}$, a small perturbation of the metric can induce a large error for the final state $(y(T),\d _{t}y(T))$. In other words, the two dynamics are no longer close. In particular, the map 
 \begin{equation*}
 \X^k(\M) \owns (\k, g)  \longrightarrow H_{\k, g} \in \mathscr L \big(H^1(\M)
 \oplus  L^2(\M),  L^2((0,T)\times \M)\big)
 \end{equation*}
 is not continuous. 
 \end{remark}
 \begin{remark}
   The result of Theorem~\ref{HUM operator} can also be stated on open
   bounded smooth domains of $\R^n$ in the case of homogeneous
   Dirichlet condition. In fact, as can be checked in what follows,
   its proof only relies on basic properties of microlocal defect
   measures (support localization and propagation) that are known to be valid in
   this framework;  see~\cite{Burq}.
 \end{remark}
 \begin{remark}
   \label{rem: HUM operator - GCC}
    In the statement of Theorem~\ref{HUM operator} if the
    neighborhood $\mathcal{U}$ of $ (\k, g)$ in $\X^k$ is small
    enough, the pair $(\omega, T)$ also satisfies the geometric
    control condition of Definition~\ref{def: Geometric control
      condition-thm} for $(\tk, \tg)$ and therefore
    $f ^{y^0,y^1}_{\tk, \tg}$ is well defined. In particular, this is
    clear as in the case $k \geq 2$ there is a well defined and unique
    geodesic flow.

    \medskip
    The case $k =1$ is quite different as there is no
    geodesic flow, as already mentioned above. However, given $(\k,g)
    \in \X^1$ and $(\omega, T)$ if the
    Rauch-Taylor geometric control condition of Definition~\ref{def:
      Geometric control condition-thm} holds for $(\omega, T)$ for the
    geodesics associated with $g$,
    given any \nhd $\mathcal U$ of  $(\k,g)$ in $\X^1$ one can still
    find $(\tk,\tg) \in \mathcal U$ such that 
    \begin{enumerate}
    \item  the  geometric
    control condition still holds for the geodesics associated with
    $\tg$,
    \item  the result of  Theorem~\ref{HUM operator} 
    also holds. 
    \end{enumerate}
  \end{remark}
\begin{theorembis}[Lack of continuity of the HUM-operator --  Case $k=1$]
  \label{HUM operator bis}
  Let $k = 1$ and $(\k, g) \in \X^1$ as above.  For any neighborhood
  $\mathcal{U}$ of $(\k, g)$ in $\X^1(\M)$, there exist
  $(\tk, \tg) \in \mathcal{U}$  and an initial
   data $(y^{0},y^{1}) \in H^1(\M) \times  L^2(\M)$,
   with $ \Norm{y^0}{H^1}^2+\Norm{y^1}{L^2}^2 =1$, such that the geometric control
  condition of Definition~\ref{def: Geometric control condition-thm}
  for geodesics given by the metric $\tg$ holds and moreover the
  results listed in Theorem~\ref{HUM operator} hold. 
\end{theorembis}
The proofs of Theorems~\ref{HUM operator} and \ref{HUM operator bis}
are given in Section~\ref{sec: proof HUM operator}.
  
\medskip
We finish this section with some remarks and some questions.
\begin{remark}
  \label{remark:}
  In all results above we have used $1_{(0,T)\times\omega}$ as a
  control operator, that is, the characteristic function of an open
  set. We could have also considered a control operator given by
  $1_{(0,T)}(t)\chi(x)$, with $\chi$ a
 smooth function on $\M$. 
  The controlled wave equation then has the form
  \begin{equation}
    \label{eq.controlint - smooth control op}
    P_{\k,g} y =  \bld{1}_{(0,T)}\, \chi  \, f, \quad (y_{|t=0}, \d_t y_{|t=0}) = (y^0, y^1),
  \end{equation}
  In such case, the open set to be used in the geometric control
  condition is $\omega = \{ \chi\neq 0\}$. This is often done this way, in
  particular since the smoothness of the function $\chi$ allows one to use some
  microlocal techniques that require regularity in the operator
  coefficients. The results and proofs of the present article can be
  written {\em mutatis mutandis} for this type of control operator.
\end{remark}

\subsubsection{Comparison with the smooth case and some open questions}
Following on the previous remark, with a smooth in space control
operator, one can wonder above the smoothness of the HUM
operator. This question is addressed in the work of the second author
jointly with G.~Lebeau~\cite{DL:09}. In fact, a gain of regularity in
the initial data $(y^0, y^1)$ yields an equivalent gain of regularity in the HUM
control function $f^{y^0,y^1}_{\k, g}$. For instance, for $(y^0, y^1)
\in H^2(\M)\times H^1(\M)$ one finds $f^{y^0,y^1}_{\k, g} \in
\Con^0([0,T], H^1(\M))$. Note that the result of \cite{DL:09} is proven
in the case of smooth coefficients, that is, $(\k,g) \in
\X^\infty$. We thus consider this smooth case in the discussion that
ends this introductory section. Open questions 
around  the results of Theorems~\ref{HUM
   operator} and \ref{HUM operator bis} are then raised.

 As we shall see in their proofs, the result of
 Theorems~\ref{HUM operator} and \ref{HUM operator bis} rely
 on the high frequency behavior of the solutions to \eqref{eq:
   stab-HUM}. In the case of smooth coefficients and a smooth control
 operator, if we assume smoother data $(y^0, y^1)$ in the HUM control
 process, the result of Theorem~\ref{HUM operator} does not hold any
 more. The HUM control process becomes regular with respect to
 $(\k,g)$ as expressed in the following proposition.
\begin{proposition}[HUM control process for smooth data]\label{Control of smooth data}
  Consider $ (\k, g) \in \X^\infty(\M)$ and let $\chi \in \Cinf(\M)$.
  Set $\omega = \{ \chi \neq 0\}$ and assume that $(\omega,
  T)$ fulfills the geometric control condition of Definition~\ref{def:
    Geometric control condition-thm} for  the geodesics associated
  with $(\k, g)$.  Let $\alpha \in
  (0,1]$. There exists $C_\alpha>0$ such that for any $(\tk, \tg) \in
  \X^\infty(\M)$ and any $(y^{0},y^{1}) \in H^{1+\alpha}(\M) \times
  H^{\alpha}(\M) $, the respective solutions $y$ and $\ty$ to
\begin{equation*}
  \begin{cases}
    P_{\k, g} y 
    =  \bld{1}_{(0,T)}\, \chi \,f ^{y^0,y^1}_{\k,  g} & \text{in} \  (0,T)\times \M,\\
    (y,\d _{t}y)_{|t=0}=(y^0,y^1) & \text{in} \ \M,
  \end{cases}
  \qquad 
  \begin{cases}
    P_{\tk, \tg} \ty 
    =  \bld{1}_{(0,T)} \, \chi  \,f ^{y^0,y^1}_{\k,  g} & \text{in} \  (0,T)\times \M,\\
    (\ty ,\d _{t}\ty )_{|t=0}=(y^0,y^1) & \text{in} \ \M,
  \end{cases}
\end{equation*}
satisfy
\begin{equation*}
  \E_{\k, g}(y-\ty)(T)^{1/2}
  \leq C_\alpha  \Norm{(\k, g) -(\tk, \tg) }{\X^1(\M)}^\alpha \, \Norm{(y^{0},y^{1}) }{H^{1+\alpha} (\M)\oplus H^\alpha(\M)}.
\end{equation*}
\end{proposition}
The proof of Proposition~\ref{Control of smooth
   data} is given in Section~\ref{sec: proof Control of smooth data}.

 \medskip In the above proposition coefficients are chosen smooth,
 quite in contrast with the rest of this article. As explained above,
 and as the reader can check in the proof, this lies in the use of the
 regularity of the HUM operator with respect to the data $(y^0,y^1)$,
 a result proven for smooth coefficients in \cite{DL:09}. The result
 of  Proposition~\ref{Control of smooth data} raises the following
 natural questions:
 \begin{enumerate}
 \item Does the HUM operator exhibit regularity with respect
   to the data $(y^0,y^1)$ similar to what is proven in \cite{DL:09} in the case of not so smooth coefficients?
 \item If so, if one increases the smoothness of the data $(y^0,y^1)$
   as in Proposition~\ref{Control of smooth data}, does the HUM control
 process also become regular with respect of the metric?
 \end{enumerate}

\section{Geometric aspects and operators}
\label{sec: Geometric facts}

We define the smooth manifold $\L = \R \times \M$ and $T^* \L$  its cotangent
bundle. We denote by $\pi: T^*\L \to \L$ the natural projection.
Elements in $T^*\L $ are denoted by $(t,x, \tau, \xi)$. One has $\pi
(t,x, \tau, \xi) = (t,x)$.

Setting $|\xi|_x^2 = g_x(\xi, \xi)$ the Riemannian norm in
the cotangent space of $\M$ at $x$, we define
\begin{align*}
S^*\L = \{(t,x,\tau,\xi) \in T^* \L, \tau^2 + |\xi|_x^2 = 1\} , 
\end{align*}
the cosphere bundle of $\L$.  We shall also use the associated cosphere bundle in the spatial  variables only,
\begin{align*}
S^* \M = \{(x, \xi) \in T^* \M, |\xi|_x^2 = 1/2\} .
\end{align*}
For a $\Con^k$-metric both $S^*\M$ and $S^* \L$ are
$\Con^k$-manifolds. 

Consider a $\Cinf$-atlas $\Atlas^\M = (\mathcal C_j^\M)_{j \in J}$ of
$\M$, $\#J< \infty$,   with $\mathcal
C^\M_j = (O_j, \theta_j)$ where $O_j$ is an open set of $\M$ and
$\theta_j: O_j \to \tilde{O}_j$ is a bijection for
$\tilde{O}_j$ an open set of $\R^d$. For $j \in J$, we define
$\mathcal C_j = (\O_j, \vartheta_j)$ with $\O_j = \R\times O_j$ and 
\begin{align*}
  \vartheta_j: \O_j  &\to \tO_j\\
  (t,x)&\mapsto \big(t, \theta_j(x)\big),
\end{align*}
with $\tO_j = \R \times \tilde{O}_j$. 
Then $\Atlas = (\mathcal C_j)_{j \in J}$ is a $\Cinf$-atlas for $\L$.

In what follows for simplicity we shall use the same notation for an
element of $T^* \L$ and its local representative if no confusion arises.

\subsection{Hamiltonian vector field and \bichars}
\label{sec: Hamiltonian vector field and bichars}


Let  $(\k, g) \in \X^k$, $k=1$ or $2$.
The principal symbol of the wave operator $P_{\k,g}$ is given by 
\begin{equation}
  \label{eq: def p}
  p(t,x, \tau, \xi)  = p_{\k,g}(t,x, \tau, \xi) = - \tau^2 + |\xi|_x^2 , \qquad 
  (t,x, \tau, \xi) \in  T^*\L .
\end{equation}
In local charts, one has
\begin{equation*}
  p (t,x, \tau, \xi) = -\tau^2 + \sum_{1\leq i,j\leq d} g^{ij}(x) \xi_i \xi_j.
\end{equation*}
Note that $(g^{ij}(x))_{i,j}$ is the inverse of $(g_{ij}(x))_{i,j}$,
the latter being the local representative of the metric.

We denote by $\Hp$ the Hamiltonian vector field associated with $p$, that is, the unique vector field
such that $ \{p, f\} = \Hp f$ for any smooth function $f$. 
Here, $\{.  ,.\}$ denotes the Poisson bracket, that is, in local chart
\begin{align}
  \{ p, f\} =\d_\tau p \ \d_t f - \d_t p\  \d_\tau f + 
  \sum_{1\leq j \leq d} ( \d_{\xi_j}p \ \d_{x_j}f - \d_{x_j}p\  \d_{\xi_j}f ),
\end{align}
yielding 
\begin{align*}
  \Hp = -2 \tau  \d_t + \nabla_\xi p \cdot \nabla_x
  - \nabla_x p \cdot \nabla_\xi,
\end{align*}
as $p$ is in fact  independent of the time variable $t$. The
Hamiltonian vector field $\Hp$ is of class $\Con^{k-1}$.
Observe that, for a function $f$ of the variables $(t,x,\tau,\xi)$,
one has 
\begin{align*}
  \transp{\Hp} f = 2 \tau \d_t f  - \div_x (f \nabla_\xi p) 
  +\div_\xi  (f \nabla_x p), 
\end{align*}
with which one deduces 
\begin{align}
  \label{eq: transp Hp}
  \transp{\Hp}  = - \Hp,
\end{align}
even in the case $(\k, g) \in \X^1$. 

\bigskip
First, consider the case $k =2$. Thus, $\Hp$ is a $\Con^1$-vector field. 
For $\y \in T^*\L$ one denotes by $s \mapsto \phi_s(\y)$
the unique maximal solution to 
\begin{equation}
  \label{eq: hamiltonian flow}
  \frac{d}{ds}\phi_s(\y) = \Hp \phi_s(\y), \ \  s\in \R, \quad 
  \text{and} \ \ \phi_{s=0}(\y) = \y,
\end{equation}
as given by the Cauchy-Lipschitz theorem.
One calls $(s,\y) \mapsto \phi_s(\y)$ the Hamiltonian flow map.
Let $s \mapsto \gamma(s)$ be an integral curve of $\Hp$, that is,
$\gamma(s) = \phi_s(\y)$ for some $\y\in T^*\L$. For any smooth
function $f$ on $T^*\L$ one has
\begin{align*}
\frac{d}{d s}  f \circ \gamma(s) = \Hp f \big(\gamma(s)\big),
\end{align*}
Note that $\Hp \tau=0$, meaning that the variable $\tau$ is
constant along $\gamma$.
Note also  that the value of $p$ remains constant along $\gamma$
since $\Hp p = \{ p, p\}=0$. Hence, $|\xi|_x^2 = g_x(\xi,\xi)$ is also
constant. Thus, if $\gamma(0) \in S^*\L$ then $\gamma(s)$ remains in
$S^*\L$, and for $\y\in S^*\L$, the  vector field $\Hp$ at $\y$ is
tangent to $S^*\L$. Consequently, we may  consider $\Hp$ as a tangent vector
field on the $\Con^2$-manifold $S^*\L$.
In particular $\Hp a$ makes sense  if
$a \in \Con^{1}_c (S^*\L)$. If moreover $a \in \Con^{2+ \ell}_c
(S^*\L)$, $\ell \geq 0$, one has $\Hp a \in \Con^{1}_c (S^*\L)$.

Since $\Hp p=0$, the flow $\phi_s$ preserves $\Char(p)= p^{-1} (\{0\})$,
the characteristic set of $p$.  As is done classically, we call
\bichar an integral curve for which $p=0$. Observe then that~\eqref{eq:
  hamiltonian flow} defines a flow on the $\Con^2$-manifold
$$
\Char(p) \cap S^*\L = \{(t, x, \tau, \xi);\ \tau^2 = 1/2 \ \text{and} \ |\xi|_x^2 = 1/2\} .
$$

\bigskip Second, consider the case $k=1$. Then $\Hp$ is only a
continuous vector field. Thus, for any $\y \in \Char(p)$ there exists
a maximal \bichar $s \mapsto \gamma(s)$ defined on $\R$ such that
$\gamma(0) = \y$, that is,
\begin{align*}
  \frac{d}{d s} \gamma(s) = \Hp\big(\gamma(s)\big), \qquad s \in \R,
\end{align*}
by the Cauchy-Peano theorem. Uniqueness is however not guaranteed
and the notion of flow cannot be used in the case $k=1$. Since the
value of $|\xi|_x$ remains constant and  the manifold $\M$ is compact,  maximal bicharacteristics are actually defined {\em globally}.

As above, if $\gamma(0) \in S^*\L$ (\resp $\Char(p) \cap S^*\L$) one
has $\gamma(s) \in  S^*\L$ (\resp $\Char(p) \cap S^*\L$) for all
$s \in \R$. The hamiltonian vector field $\Hp$ can be viewed as a
$\Con^0$-vector field on the $\Con^1$-manifold $S^*\L$ (\resp on the
$\Con^1$-manifold $\Char(p) \cap S^*\L$).  For
$a \in \Con^{1+\ell}_c (S^*\L)$, $\ell\geq 0$, one finds $\Hp a \in \Con^0_c (S^*\L)$.

\bigskip
Finally, connection between \bichar and geodesics can be made.
For this we recall that if $\xi \in T_x^*\M$ for some $x \in \M$ one can define
$v \in T_x\M$ by $v = \xi^\sharp$, which  reads in local coordinates
$v^i = \sum_j g^{ij}(x) \xi_j$. In particular $|v|_x^2= g_x(v,v) = |\xi|_x$.
If now $\y^0 = (t^0, x^0, \tau^0, \xi^0 )\in \Char(p) \cap S^*\L$ and
letting $s \mapsto \y(s) = \big(t(s), x(s), \tau, \xi(s) \big)$ be a
\bichar such that $\y(0) = \y^0$. One has
$\tau = \tau^0$ and $t(s) = t^0 - 2 \tau^0 s$. The map
\begin{align*}
  X: t \mapsto x \big( (t^0 - t)/ (2 \tau^0)\big),
\end{align*}
can be proven to be the geodesic originating from $x^0$ in the
direction given by $v^0 = (\xi^0)^\sharp$ and parametrized
by $t$.

We now compute the speed at which the geodesic is traveled.
We have $\frac{d X}{ d t} (t) = \frac{-1}{2 \tau^0} \frac{d
  x(s)}{ d s}$, which yields
\begin{align*}
 \frac{d X}{ d t} (t)  = \frac{-1}{2 \tau^0} \nabla_\xi p \big(x(s), \xi(s)\big)
  = - \frac{\xi(s)^\sharp}{\tau^0}.
\end{align*}
It follows that
\begin{align*}
  |\frac{d X}{ d t} (t)|_x = |\xi(s) ^\sharp|_x /|\tau^0|
  = |\xi(s)|_x /|\tau^0| = |\xi^0|_x /|\tau^0| 
  =1,
\end{align*}
since $\y^0 \in \Char(p)$. 
Hence, the projection of the \bichar $s \mapsto \gamma(s)$ yields a
geodesic traveled at speed 1.

\subsection{Geometric control condition}
\label{sec: Geometric control condition}

As the projections of \bichars onto  $\L$ yield geodesics,
in the case $k\geq 2$, we can
state the Rauch-Taylor geometric control condition
\cite{Rauch-Taylor} formulated in  Definition~\ref{def: Geometric
  control condition-thm} with the notion of Hamiltonian flow
introduced above.
\begin{definitionbis}[geometric control condition, $k \geq 2$]
  \label{def: Geometric control condition-bis}
  Let $g$ be a $\Con^2$ metric 
  and let $\omega$ be an open set of $\M$ and $T>0$. One says that
  $(\omega, T)$ fulfills the geometric control condition if
  for all $\y  \in \Char(p)$
 one has $\pi \big(\phi_s(\y)\big) \in (0, T) \times \omega$
  for some $s \in \R$.  
\end{definitionbis}

In the case $k=1$,  since  $g$ is only $\Con^1$ there
is no flow in general, one rather writes the geometric control condition by means
of maximal \bichars.
\begin{definitionter}[generalized geometric control condition, $k =1$]
  \label{def: Geometric control condition-ter}
  Let $g$ be a $\Con^1$ metric and let $\omega$ be an open set of $\M$ and $T>0$. One says that
  $(\omega, T)$ fulfills the geometric control condition if
  for {\em any} maximal \bichar $s \mapsto \gamma(s)$ in $\Char(p)$  
 one has $\pi \big( \gamma(s)\big) \in (0, T) \times \omega$
  for some $s \in \R$.  
\end{definitionter}

In other words, for all $\y \in \Char(p)$, all \bichars
that go through $\y$ meet the cotangent bundle above
$(0, T) \times \omega$.

Naturally, Definitions~\ref{def: Geometric control condition-bis} and \ref{def: Geometric control condition-ter} coincide in the case $k = 2$ because
of the uniqueness of a \bichar going through a point of  $\Char(p)$.

\subsection{Symbols and pseudo-differential operators}
\label{sec: measures symbols operators}

Here, we follow \cite[Section 1.1]{Bu97} for the notation. We denote
by $H^k(X )$ or $H^k_{\loc}(X )$, with $X = \M$ or $ \L$, the usual
Sobolev space for complex valued functions, endowed with its natural
inner product and norm. In particular, the $L^2(X )$-inner product is
denoted by $(.,.)_{L^2(X )}$.

Classical polyhomogeneous symbol classes on $T^* \R^n\simeq \R^n\times \R^n$ are denoted by
$S_{\phg}^m(\R^n\times \R^n)$ and the classes of associated operators
by $\Psi_{\phg}^m(\R^n)$.  We recall that symbols
in the class $S_{\phg}^{m}(\R^n\times \R^n)$ behave well
with respect to changes of variables, up to symbols in
$S_{\phg}^{m-1}(\R^n\times \R^n)$; see \cite[Theorem~18.1.17 and Lemma~18.1.18]{Hormander-III}.

We define $S_{c, \phg}^{m}(T^*\L)$  as the set of  polyhomogeneous symbols of order $m$
on $T^*\L$ with compact support in the variables $(t,x) \in \L$ (note
that compactness with respect to $x \in \M$ is obvious). Having the
manifold $\M$ smooth is important for symbols and  following
pseudo-differential operators to be simply defined. 

For any $m$, the restriction to the sphere
\begin{equation}
  \label{eq: identification symbols}
  S_{c, \phg}^{m}(T^*\L) \to \Cinfc(S^*\L), \quad a \to a_{|S^*\L},
\end{equation}
is onto. This allows one to identify a homogeneous symbol with a
smooth function on $S^*\L$ with compact support.

We denote by $\Psi_{c, \phg}^{m}(\L)$  the space of polyhomogeneous pseudo-differential operators of order
$m$ on $\L$: one says that $Q \in \Psi_{c, \phg}^{m}(\L)$ if $Q$ maps
$\Cinfc(\L)$ into $\mathscr D' (\L)$ and 
\begin{enumerate}
  \item its kernel $K(x,y) \in \D'(\L \times \L)$ is such that 
    $\supp (K)$ is compact in $\L \times\L$;
  \item $K(x,y)$ is smooth away from the diagonal $\Delta_{\L} = \{
    (t,x;t,x);\ (t,x) \in \L\}$;
  \item for any local chart  $\mathcal C_j = (\O_j, \vartheta_j)$ and all $\phi_0$, $\phi_1
    \in \Cinfc(\tO_j)$
    one has 
    \begin{align*}
    \phi_1 \circ\big(\vartheta_j^{-1}\big)^\ast \circ Q \circ \vartheta_j^\ast
      \circ \phi_0
      \in \Op \big(S_{c, \phg}^{m}(\R^{d+1}\times\R^{d+1})\big).
   \end{align*}
\end{enumerate}

For $Q\in \Psi_{c, \phg}^{m}(\L)$, we denote by
$\sigma_m(Q) \in S_{c, \phg}^{m}(T^* \L)$ the
principal symbol of $Q$; see~\cite[Chapter 18.1]{Hormander-III}. Note
that the principal symbol is uniquely defined in $S_{c, \phg}^{m}(T^* \L
)$ because of the polyhomogeneous structure (see the
remark following Definition~18.1.20 in~\cite{Hormander-III}).  The
application $\sigma_m$ enjoys the following properties.
\begin{enumerate}
\item The map $\sigma_m: \Psi_{c, \phg}^{m}(\L) \to
  S_{c, \phg}^{m}(T^* \L)$ is onto.
\item For all $Q \in \Psi_{c, \phg}^{m}(\L)$,
  $\sigma_m(Q) = 0$ if and only if $Q \in \Psi_{c, \phg}^{m-1}(\L)$.
\item For all $Q\in \Psi_{c, \phg}^{m}(\L)$,
  $\sigma_m(Q^*) = \overline{\sigma_m(Q)}$.
\item For all $Q_1 \in \Psi_{c, \phg}^{m_1}(\L)$ and
  $Q_2 \in \Psi_{c, \phg}^{m_2}(\L)$, one has $Q_1 Q_2
  \in \Psi_{c, \phg}^{m_1+m_2}(\L)$ with
  \begin{align*}
  \sigma_{m_1 +m_2}(Q_1 Q_2) = \sigma_{m_1}(Q_1)\sigma_{m_2}(Q_2).
  \end{align*}
\item For all $Q_1 \in \Psi_{c, \phg}^{m_1}(\L)$ and $Q_2 \in
  \Psi_{c, \phg}^{m_2}(\L)$, one has $[Q_1 ,Q_2] = Q_1 Q_2 - Q_2
  Q_1 \in \Psi_{c, \phg}^{m_1+m_2 - 1}(\L)$ with
  \begin{align*}
  \sigma_{m_1 +m_2 - 1}([Q_1 ,Q_2] ) = \frac{1}{i} \{ \sigma_{m_1}(Q_1) , \sigma_{m_2}(Q_2) \}.
  \end{align*}
 
\item If $Q \in \Psi_{c, \phg}^{m}(\L)$, then $Q$ maps
  continuously $H^k_{\loc}(\L)$ into $H_{\comp}^{k-m}(\L)$. In
  particular, for $m < 0$, $Q$ is compact on $L_{\loc}^2(\L)$.
\end{enumerate}

\medskip
Given an operator $Q \in \Psi^m_{c, \phg}(\L)$, one sets
\begin{align*}
  \Char(Q) = \Char\big(\sigma_m(Q)\big)
  = \{\y \in T^*\L , \sigma_m(Q)(\y) = 0\}.
\end{align*}

\section{Microlocal defect measure and propagation properties}
\label{sec: mdm and pde}
A defect measure is used to characterize locally the failure of a
sequence to strongly converge, meaning some concentration
phenomenon. This characterization can be made finer by further considering
microlocal concentration phenomena. 

\subsection{Microlocal defect density measures}
We define $\M_+ (S^*\L)$ as the set of positive
density measures on $S^*\L$.  For
$\mu \in \M_+ (S^*\L)$ and $a \in \Con_c^0(S^*\L)$, we shall write
\begin{align*}
\dup{\mu}{a}_{S^*\L} = \int_{S^*\L} a(\y) \mu (d \y) ,
\end{align*}
for the duality bracket. This notation will also be used for $a
\in S_{c, \phg}^0(T^*\L)$  according to the identification map~\eqref{eq:
  identification symbols}.

\medskip
Consider a sequence $(u^k)_{k \in \N}\subset L^2_{\loc}(\L)$ that
converges weakly to $0$.  Here, to define the $L^2$-norm and inner
product on $\L$ we use a fixed
$(\k^0,g^0)$  chosen in $\X^1(\M)$; see \eqref{eq: L2-inner product and norm}.

As a consequence of \cite[Theorem 1]{Ge91}, there exists a subsequence
of $(u^k)_{k \in  \N} $ (still denoted by $(u^k)_{k \in \N}$ in what
follows) and a density measure
$\mu  \in \M_+(S^*\L)$, 
such that 
\begin{equation}
  \label{eq: exists mdm mu 0}
  \lim_{k \to \infty} \dup{Q u^k}{\ovl{u^k}}_{L^2_{\comp}(\L), L^2_{\loc}(\L)} 
  =  \dup{\mu}{\sigma_0(Q)}_{S^*\L},
\end{equation}
for any $Q \in \Psi^0_{c, \phg}(\L)$.
Recall that symbols in $S^0_{c, \phg}(T^* \L)$ are compactly supported in time $t$
here. 
We also refer to \cite{Tartar90} and \cite{Bu97}. One
calls $\mu$ a microlocal defect (density) measure associated with
$(u^k)_{k\in \N}$. 


Similarly, one can use the notion of $H^1$-microlocal defect
density measure. Consider $(u^k)_{k \in \N}\subset H^1_{\loc}(\L)$ that
converges  weakly to $0$.  Then, there exists a subsequence of
$(u^k)_{k \in  \N} $ (still denoted by $(u^k)_{k \in \N}$ ) and a density
measure $\mu  \in \M_+(S^*\L)$
such that for any $Q \in \Psi^2_{c, \phg}(\L)$ 
\begin{equation}
  \label{eq: exists mdm mu 1}
  \lim_{k \to \infty} \dup{Q u^k}{\ovl{u^k}}_{H_{\comp}^{-1}(\L), H_{\loc}^1(\L)} 
  =  \dup{\mu}{\sigma_2(Q)}_{S^*\L}.
\end{equation}

Naturally, in either cases, the density measure $\mu$ depends on the choice
made of $(\k^0,g^0) \in \X^1(\M)$. In what follows we shall make clear
what choice is made.

\subsection{Local representatives}
\label{sec: Local representatives of the measure}
Consider a finite atlas $\Atlas = (\mathcal C_j)_{j \in J}$ on $\L$,
as introduced in Section~\ref{sec: Geometric facts}, with
$\mathcal C_j = (\O_j, \vartheta_j)$. Consider a smooth partition of
unity $(\chi_j)_{j\in J}$ subordinated to the covering by the open
sets $(\O_j)_j$.  We consider also $\tchi_j, \hchi_j \in \Cinf(\L)$
supported in $\O_j$ such that $\tchi_j \equiv 1$ on a \nhd of
$\supp(\chi_j)$ and $\hchi_j \equiv 1$ on a \nhd of $\supp(\tchi_j)$.
Set also $\chi_j^{\mathcal C_j} = (\vartheta_j^{-1})^* \chi_j$,
$\tchi_j^{\mathcal C_j}= (\vartheta_j^{-1})^* \tchi_j$, and
$\hchi_j^{\mathcal C_j}= (\vartheta_j^{-1})^* \hchi_j$.  One has
$\chi_j^{\mathcal C_j}, \tchi_j^{\mathcal C_j}, \hchi_j^{\mathcal C_j}
\in \Cinfc(\tO_j)$, with $\tO_j = \vartheta_j(\O_j)$.

Let $(u^k)_k \subset H^1_{\loc}(\L)$ that converges weakly to $0$, 
$Q \in \Psi_{c, \phg}^{2}(\L)$, and $j \in J$. 
One can write
\begin{align*}
  \chi_j Q = \chi_j Q \tchi_j  + \chi_j Q (1 - \tchi_j).
\end{align*}
Since $\chi_j Q (1 - \tchi_j)$ is a regularizing operator one finds
\begin{align*}
  \dup{\mu}{ \chi_j \sigma_2(Q)}_{S^*\L}
  \ \ &\mathop{\sim}\ \ 
  \dup{\chi_j Q u^k}{\ovl{u^k}}_{H_{\comp}^{-1}(\L), H_{\loc}^1(\L)}\\
   &\mathop{\sim}\ \ 
  \dup{\chi_j \tchi_j Q  \tchi_j v _j^k}{\ovl{v_j^k}}_{H_{\comp}^{-1}(\L),H_{\loc}^1(\L)}, 
  \qquad \text{as} \ k \to +\infty ,
\end{align*}
for $v_j^k = \hchi_j u^k$.

The operator
$Q_j = (\vartheta_j^{-1})^* \tchi_j Q \tchi_j (\vartheta_j)^*$ is a
pseudo-differential operator of order $2$ on $\R^{d+1}$ with principal
symbol $q_j = \tchi_j^2 q^{\mathcal C^j}$, where $q^{\mathcal C^j}$  is
the local representative of $\sigma_2(Q)$.
Set also  $v^{k,\mathcal C_j}_j=  (\vartheta_j^{-1})^*   v^k_j$. It converges weakly to $0$ in
$H^1(\R^{d+1})$. Associated with this sequence is a microlocal defect
measure $\mu_j$. If one writes 
\begin{align*}
  \dup{\chi_j \tchi_j Q  \tchi_j v_j^k}{\ovl{v_j^k}}_{H_{\comp}^{-1}(\L),H_{\loc}^1(\L)}
  = \dup{\chi^{\mathcal C^j}_j Q_j v^{k,\mathcal C_j}_j}  {\ovl{v^{k,\mathcal C_j}_j}}_{H_{\comp}^{-1}(\R^{d+1}),H_{\loc}^1(\R^{d+1})},
\end{align*}
one obtains
\begin{align*}
  \dup{\mu}{ \chi_j \sigma_2(Q)}_{S^*\L}
  &= \dup{\mu_j}{\chi^{\mathcal C^j}_j q_j}_{S^* \tilde{O}^{j}}
    = \dup{\mu_j}{\chi^{\mathcal C^j}_j q^{\mathcal C^j}}_{S^* \tilde{O}^{j}}.
\end{align*}
Note that here, the $L^2$ and $H^s$-norms on $\R^{d+1}$ are based on
the local representative of the  density measure $\k^0 \mu_{g^0} dt$.
One thus sees that the local representative of $\chi_j \mu$ is precisely
$\chi_j^{\mathcal C^j} \mu_j$, that is,
$\chi_j \mu = \vartheta_j^* \big(\chi_j^{\mathcal C^j} \mu_j \big) =
\chi_j \vartheta_j^*\mu_j$. 
Summing up, we thus have
\begin{align*}
  \mu = \sum_{j\in J} \chi_j  \mu = \sum_{j\in J} \chi_j \vartheta_j^*\mu_j.
\end{align*}
and
\begin{align*}
  \dup{\mu}{\sigma_2(Q)}_{S^*\L}
  &= \sum_{j\in J} \dup{\mu}{ \chi_j \sigma_2(Q)}_{S^*\L}
    =\sum_{j\in J} \dup{\mu_j}{\chi^{\mathcal C^j}_j q^{\mathcal C^j}}_{S^* \tilde{O}^{j}}.
\end{align*}

\begin{remark}\label{local-global}
  Local properties of microlocal defect measures like $\mu$ can be
  deduced from the properties of $\chi_j^{\mathcal C^j} \mu_j$.  In
  what follows most results are of local nature. In such cases we
  shall work in local charts and use Sections~\ref{sec: measures
    symbols operators} and \ref{sec: Local representatives of the
    measure} to bring the analysis to open domains of $\R^ {d+1}$.
\end{remark}
 
\subsection{Operators with a low regularity}

An important tool we use to handle low regularity terms in what follows is a
result due to R.~Coifman and Y.~Meyer (see \cite[Proposition
IV.7]{Coifman-Meyer}) and some of its consequences that we list below.
\begin{theorem}[Coifman-Meyer]\label{C-M}
  Let $Q \in \Psi_{\phg}^1(\R^n \times \R^n)$. If
  $m \in W^{1,\infty}(\R^n)$ the
  commutator $[Q, m]$ maps $L^2(\R^n)$
  into itself continuously. Moreover there exists $C>0$ such that 
  \begin{align*}
    \bigNorm{[Q,m]}{L^2\rightarrow L^2} \leq C \Norm{m}{W^{1,\infty}},
    \qquad m \in W^{1,\infty}(\R^n).
  \end{align*}
 \end{theorem}
 We deduce the following corollary.
 \begin{corollary}
   \label{cor: peu-reg}
   Let $Q \in \Psi_{\phg}^1(\R^n \times \R^n)$ be such that its kernel
   has compact support in $\R^n \times \R^n$. With $q \in
   S_{\phg}^{1}(\R^n \times \R^n)$ its principal symbol.
    
   Let $m \in \Con^1(\R^n)$.  There exist $K_1$ and
   $K_2$, compact operators on $L^2(\R^n)$, with compactly supported
    kernels, such that 
    \begin{align}
      \label{eq: cor peu-reg}
      [Q, m] 
      = \frac{1}{i} \nabla_{x}m \cdot \Op (\nabla_{\xi} q)
      +   K_1
      = \frac{1}{i} \Op (\nabla_{\xi} q) \cdot \nabla_{x}m 
      +   K_2.
    \end{align}
\end{corollary}
\begin{proof}
Consider a sequence $(m^k)_{k\in \N}
\subset \Cinf(\R^n)$ such that
\begin{align*}
  \sum_{|\alpha|\leq 1} \Norm{\d_x^\alpha (m^k - m)}{L^\infty}  \ \
  \to \ \ 0\qquad \text{as} \ k \to +\infty.
\end{align*}
Classical symbolic calculus gives
\begin{align}
  \label{eq: cor peu-reg-1}
  [Q, m^k]=\frac{1}{i} \nabla_{x}m^k \cdot \Op (\nabla_{\xi} q) + K_1^k,
\end{align}
with $K_1^k = \Op (r_1^k)$ for some $r_1^k \in
S_{\phg}^{-1}$, $j=1,2$. Thus,  $K_1^k$ is  bounded from $L^2(\R^n)$ into
$H^1(\R^n)$. In addition, since $K_1^k$ has a  kernel with
compact supports in $\R^n \times \R^n$, it is compact on  $L^2(\R^n)$.
Note that the support of the  kernel of $K_1^k$ lies in a compact $\mathcal K$ of
$\R^n \times \R^n$ that is uniform with respect to $k$.

On the other hand, observe that
\begin{align*}
  \nabla_{x}m^k \cdot \Op (\nabla_{\xi} q)  \ \ \to \ \ \nabla_{x}m \cdot \Op (\nabla_{\xi} q) 
  \ \ \text{in} \ {\mathscr L}(L^2(\R^n)).
\end{align*}
Moreover, from Theorem~\ref{C-M} applied to $m^k - m$, one also has
\begin{align*}
  [Q, m^k] \ \ \to \ \ [Q, m] 
  \ \ \text{in}\  {\mathscr L}(L^2(\R^n)).
\end{align*}
Using  then \eqref{eq: cor peu-reg-1} we deduce that $(K_1^k) _{n\in \N}$
converges to some $K^1$ in ${\mathscr L}(L^2(\R^n))$, and from the closedness of 
the set of compact operators  in  ${\mathscr L}(L^2(\R^n))$ we find that
$K^1$ is compact. Moreover, $K^1$ has a kernel supported in $\mathcal K$.
The limits above give the first equality in \eqref{eq: cor peu-reg}. The second equality follows similarly. 
\end{proof}

\bigskip
Let $\Omega$ be a bounded open set of $\R^n$ and $(\k^0,g^0) \in
\X^1(\Omega)$, with definition adapted from that of $\X^1(\M)$. The
$L^2$-inner product and norm are given by the density $\k^0
\mu_{g^0}$.
The following result is also a consequence of Theorem~\ref{C-M}.
\begin{proposition}
  \label{prop: peu-reg-2}
Let $(u^k)_{k\in \N}\subset H_{\loc}^1(\Omega)$ be a sequence that
converges weakly  to $0$ and let $\mu$ be an
$H^1$-microlocal defect density measure on $S^{*} \Omega$ associated with
the sequence $(u^k)_k$.

Let $b_1 \in W^{1,\infty}(\R^n)$ and $b_2 \in \Con^0(\R^n)$. Consider
also $Q_1, Q_2 \in \Psi_{\phg}^{1}(\R^n)$, both with kernels compactly
supported in $\Omega\times \Omega$, with
$q_1, q_2\in S_{\phg}^{1}(\R^n\times \R^n)$ for respective principal
symbol. Then, one has
\begin{align}
  \label{CM-precise}
  \dup{b_1 Q_1  \, b_2 \, Q_2u^k}{\ovl{u^k}}_{H^{-1}_{\comp} (\Omega), H^1_{\loc}(\Omega)}
  \ \ \mathop{\longrightarrow}_{k\to + \infty}  \dup{\mu}{b_1 b_2 q_1 q_2}_{S^*\Omega}.
\end{align}
More generally, assume that $(b_1^k)_{k \in \N}\subset
W^{1,+\infty}(\R^n)$ and  $(b_2^k)_{k \in \N}\subset
L^\infty(\R^n)$, 
and $(\k_k,g_k)_{k \in \N}\subset \Y(\Omega)$ with 
\begin{align*}
  \Norm{b_1^k - b_1}{W^{1,+\infty}(\R^n)}
  + \Norm{b_2^k - b_2}{L^\infty(\R^n)} 
  + \Norm{(\k_k,g_k) - (\k^0,g^0) }{\Y(\Omega)} \to 0,
  \ \ \text{as} \ \ k \to +\infty.
\end{align*}
 Then
\begin{align}
  \label{CM-precise2}
  \dup{b_1^k Q_1  \, b_2^k\, Q_2u^k}{\ovl{u^k}}_{H^{-1}_{\comp} (\Omega, \k_k \mu_{g_k}), H^1_{\loc}(\Omega, \k_k \mu_{g_k})}
  \ \ \mathop{\longrightarrow}_{k\to + \infty}  \ \  \dup{\mu}{b_1 b_2 q_1 q_2}_{S^*\Omega}.
 \end{align}
\end{proposition}
\begin{remark}
  \label{rem: Lipschitz in prop peu-reg-2}
Note that $b_1$ is chosen in $W^{1,\infty}(\R^n)$ because one cannot multiply
an element in  $H^{-1}$ by a bounded function. One derivative is
needed. 
\end{remark}
\begin{proof}[Proof of Proposition~\ref{prop: peu-reg-2}]
  With Lemma~\ref{equivalence-densites} below we may replace the density $\k_k \mu_{g_k}$ in the
  $L^2$-inner product by $\k^0 \mu_{g^0}$ and thus in the
  $H^{-1}_{\comp}$-$H^1_{\loc}$ duality. 

  We write 
  \begin{align*}
    b_1^k Q_1  \, b_2^k\, Q_2 = 
    b_1 Q_1  \, b_2\, Q_2 + R^k, \qquad 
     R^k = b_1 Q_1  \, (b_2^k-b_2) \, Q_2
    + (b_1^k-b_1) Q_1  \, b_2^k\, Q_2. 
  \end{align*}
  Note that $R^k$ maps $H^1_{\loc} (\Omega)$ into
  $H^{-1}_{\comp} (\Omega)$ continuously. Moreover because of the
  convergences of $b_1^k$ and $b_2^k$, and the boundedness of
  $(u^k)_{k \in \N}$ in $H^1_{\loc} (\Omega)$ one finds that
  $R^k u^k \to 0$ strongly in $H^{-1}_{\comp} (\Omega)$.
  Thus we can write
\begin{align*}
    \dup{ b_1^k Q_1  \, b_2^k\, Q_2  u^k}{\ovl{u^k}}_{H^{-1}_{\comp} (\Omega), H^1_{\loc}(\Omega)}
    \ \ =  
  \dup{ b_1 Q_1  \, b_2 \, Q_2  u^k}{\ovl{u^k}}_{H^{-1}_{\comp} (\Omega), H^1_{\loc}(\Omega)} + o(1)_{k\rightarrow + \infty}
\end{align*}
and  \eqref{CM-precise2} follows if we prove \eqref{CM-precise}.

According to  Theorem~\ref{C-M} the commutator $[b_1, Q_1]$ is bounded on
  $L^2(\Omega)$ implying that  $[b_1, Q_1]  \, b_2\, Q_2u^k$ is bounded in
  $L^2(\Omega)$ yielding
\begin{align*}
    \dup{[b_1, Q_1]  \, b_2 \, Q_2u^k}{\ovl{u^k}}_{H^{-1}_{\comp} (\Omega), H^1_{\loc}(\Omega)}
    = \inp{[b_1, Q_1]  \, b_2 \, Q_2u^k}{u^k}_{L^2(\Omega)}
   \ \ \mathop{\longrightarrow}_{k\to + \infty}  0,
 \end{align*}
since $u^k \to 0$ strongly in $L^2(\Omega)$. 
We may thus assume that $b_1=1$ without any loss of generality.

Let $\eps>0$ and let $b_2^\eps \in \Cinf(\Omega)$ be such that 
$\Norm{b_2  - b_2^\eps}{L^\infty}\leq \eps$. 
Write 
\begin{align*}
  Q_1  \, b_2\, Q_2
  = Q_1  \, b_2^\eps\, Q_2 
  + R^\eps, \qquad 
  R^\eps=  Q_1  \, (b_2 - b_2^\eps)\, Q_2.
  \end{align*}
One has 
$\vert\dup{R^\eps u^k}{\ovl{u^k}}_{H^{-1}_{\comp}
   (\Omega), H^1_{\loc}(\Omega)}\vert \leq C\eps$, and this leads to 
\begin{align}
 \label{eq: proof CM-precise-1} 
  \dup{ Q_1  \, b_2\, Q_2 u^k}{\ovl{u^k}}_{H^{-1}_{\comp}
   (\Omega), H^1_{\loc}(\Omega)} 
 =  
\dup{ Q_1  \, b_2^\eps\, Q_2  u^k}{\ovl{u^k}}_{H^{-1}_{\comp}
   (\Omega), H^1_{\loc}(\Omega)} 
  + 
  o(1)_{\eps \rightarrow 0}+  o(1)_{k\rightarrow + \infty}.
 \end{align}
Since $b_2^\eps$ is smooth, by symbolic calculus one has 
\begin{align}
  \label{eq: proof CM-precise-2} 
   \dup{ Q_1  \, b_2^\eps \, Q_2u^k}{\ovl{u^k}}_{H^{-1}_{\comp}
   (\Omega), H^1_{\loc}(\Omega)} 
    \ \ \mathop{\longrightarrow}_{k\to + \infty}  \ \ 
  \dup{\mu}{ b_2^\eps q_1 q_2}_{S^*\Omega}.
 \end{align}
Finally, since $\dup{\mu}{b_2^\eps q_1 q_2}_{S^*\Omega} \to \dup{\mu}{b_2 q_1
  q_2}_{S^*\Omega}$ as $\eps \to 0$,  with \eqref{eq: proof
  CM-precise-1} and \eqref{eq: proof CM-precise-2}  one concludes that
\eqref{CM-precise} holds. 
\end{proof}
\begin{lemma} 
  \label{equivalence-densites}
  Assume that $\Norm{(\k_k,g_k) - (\k^0, g^0)}{\Y(\Omega)} \to 0$ and
  consider a sequence $(f_k, h_k)_{k \in \N}$ bounded in
  $L_{\comp}^2(\Omega) \oplus L_{\loc}^2(\Omega)$. Then
  \begin{align*}
  \inp{f_k}{h_k}_{L^2 (\Omega,\k_k \mu_{g_k})}  
  =
   \inp{f_k}{h_k}_{L^2 (\Omega)}+ o(1)_{k\rightarrow + \infty}.
  \end{align*}
  If $(f_k, h_k)_{k \in \N}$ is bounded in
  $H^{-1}_{\comp}(\Omega) \oplus H^{1}_{\loc}(\Omega)$
  then 
   \begin{align*}
    \dup{f_k}{\ovl{h_k}}_{H^{-1}_{\comp}(\Omega,\k_k \mu_{g_k}), H^{1}_{\loc}(\Omega, \k_k \mu_{g_k}) }  
  = \dup{f_k}{\ovl{h_k}}_{H^{-1}_{\comp}(\Omega), H^{1}_{\loc}(\Omega)}+ o(1)_{k\rightarrow + \infty}.
   \end{align*}
 \end{lemma}
  Here, Lemma~\ref{equivalence-densites} is written in the case of a bounded  open set of the 
 Euclidean space but the same result holds in the case of a compact manifold. 
\begin{proof}
  One has $\mu_{g^0} = (\det g^0)^{1/2} dx$ and $\mu_{g_k} = (\det
  g_k)^{1/2}dx$.
  Therefore  $ \k_k \mu_{g_k} = \alpha_k \k^0 \mu_{g^0}$ with 
  $\alpha_k = \frac{\k_k}{\k^0} \Bigl(\frac{\det g_k}{\det g^0}\Bigr)^{1/2}$
  and $\alpha_k \to 1$ in the Lipschitz norm. 
\end{proof}

\subsection{Measures and partial differential equations}
Microlocal defect measures associated with sequences of solutions of
partial differential equations with smooth coefficients can have
properties such as support localization in the characteristic set and
invariance along the Hamiltonian flow. With the material developed
above, we extend these results to the case of $\Con^1$-coefficients.
We focus on the case of wave operators.

\begin{proposition}
  \label{prop-int}
  Let  $(\k^0, g^0 )\in \X^1(\M)$ and set $P^0 = P_{\k^0,g^0}$. Denote by  $p^0(x,\tau,\xi)  = - \tau^2 +
  g_x^0(\xi,\xi)$ its principal symbol.
Let $(\k_k, g_k )_{k\in \N} \subset \Y(\M)$
be such that $\Norm{(\k_k, g_k ) - (\k^0, g^0 )}{\Y(\M)} \to 0$ as $k \to
+\infty$ and set $P_k = P_{\k_k,g_k}$.

Consider a sequence $(u^k) _{k\in \N} \subset H_{\loc}^1(\L)$ that
converges to $0$ weakly and $\mu$ an $H^1$-microlocal defect density measure
associated with $(u^k) _{k\in \N}$.

Let $T_1 < T_2$. The following properties hold.
\begin{enumerate}
\item If $P_k u^k \to 0$ strongly in $H^{-1}_{\loc} \big((T_1,T_2) \times \M\big)$
then 
\begin{equation}
  \label{measure support}
  \supp (\mu) \cap S^* ((T_1, T_2)\times \M)
  \subset \Char(p^0).
 \end{equation}
 
\item If moreover $P_k u^k \to 0$ strongly in $L^2_{\loc} \big((T_1,T_2) \times \M\big)$
then one has
\begin{equation}\label{invariance}
  \transp{\Hpref} \mu =0 \ \text{in the sense of distributions on}\  S^*\big( (T_1, T_2) \times \M \big),
 \end{equation}
that is, $\dup{\mu}{\Hpref q}_{S^*\L} =0$ for all $q \in \Cinfc \big( S^*\big( (T_1, T_2) \times \M \big)\big)$.
\end{enumerate}
\end{proposition}
\setcounter{sectionbiss}{\arabic{section}}
\setcounter{propositionbiss}{\arabic{theorem}}
Since $\Hpref$ is a tangent vector field on $S^*\L$ where $\mu$ lives
(see Section~\ref{sec: Hamiltonian vector field and bichars}) note 
that $\transp{\Hpref} \mu$ makes sense in
the second item of the proposition.
Moreover note that  $\Hpref$ is a tangent vector field on $S^*\L \cap
\Char(p^0)$ and one has $\supp (\mu) \cap S^* ((T_1, T_2)\times \M)
  \subset \Char(p^0)$ by the first item of the proposition. Finally,
  notice that for a Hamiltonian vector field, $\Hpref=  -
  \transp{\Hpref}$ as recalled in Section~\ref{sec: Hamiltonian
    vector field and bichars} even in the case $(\k^0, g^0 )\in \X^1(\M)$.

  \medskip
  Naturally, Proposition~\ref{prop-int} and its proof can be adapted to the other
  energy levels. We shall also need the following result.
\begin{propositionbis}
  \label{prop-int-bis}
  With the notation of Proposition~\ref{prop-int}, consider a sequence
  $(u^k) _{k\in \N} \subset L^2_{\loc}(\L)$ that converges to $0$
  weakly and $\mu$ an $L^2$-microlocal defect density measure
  associated with $(u^k) _{k\in \N}$.

Let $T_1 < T_2$. The following properties hold.
\begin{enumerate}
\item If $P_k u^k \to 0$ strongly in $H^{-2}_{\loc} \big((T_1,T_2) \times \M\big)$
then 
\begin{equation*}
  \supp (\mu) \cap S^* ((T_1, T_2)\times \M)
  \subset \Char(p^0).
 \end{equation*}
 
\item If moreover $P_k u^k \to 0$ strongly in $H^{-1}_{\loc} \big((T_1,T_2) \times \M\big)$
then one has
\begin{equation*}
  \transp{\Hpref} \mu =0 \ \text{in the sense of distributions on}\  S^*\big( (T_1, T_2) \times \M \big).
 \end{equation*}
\end{enumerate}
\end{propositionbis}
\begin{proof}[Proof of Proposition~\ref{prop-int}]
  Consider $B \in \Psi^0_{c, \phg}(\L)$ with kernel supported in
  $\big( (T_1,T_2)\times \M\big)^2$  and $b\in S^0_{c, \phg}(\L)$ its
  principal symbol. For the definition of the  $L^2$-inner product  we
  use $(\k^0,g^0)$. We also use the partition of unity $1= \sum_{j\in J}
\chi_j$ with $\chi_j  \in \Cinfc(\O_j)$ associated with the atlas
$\Atlas$ and the additional cutoff functions
  $\tchi_j , \hchi_j \in \Cinfc(\O_j)$ that are introduced in
  Section~\ref{sec: Local representatives of the measure} and, as
  obtained in that section,  we write 
 \begin{align}
   \label{eq: localization measure property}
   \bigdup{BP_k  u^k}{\ovl{u^k}}_{H_{\comp}^{-1}(\L),
     H_{\loc}^1(\L)} 
   &= 
   \sum_{j \in J}\bigdup{\chi_j BP_k  u^k}{\ovl{u^k}}_{H_{\comp}^{-1}(\L), H_{\loc}^1(\L)} \\
   &=
   \sum_{j \in J} \bigdup{\chi_j \tchi_j BP_k \tchi_j v^k_j}
     {\ovl{v_j^k}}_{H_{\comp}^{-1}(\L), H_{\loc}^1(\L)}+ o(1)_{k\rightarrow + \infty}, \notag
 \end{align}
with $v_j^k = \hchi_j u^k$. Associated with $(\vartheta_j^{-1})^*
v_j^k$, the local representative of  $v_j^k$,  is a microlocal
defect measure $\mu_j$  in $\vartheta_j(\O_j) = \tO_j = \R\times
\tilde{O}_j$ and 
$\chi_j^{\mathcal C_j} \mu_j$ is the local representative of $\chi_j
\mu$ in this chart. See Section~\ref{sec: Local representatives of the
  measure}.

Note that we use  local representatives of the operators, functions,
and measures without introducing any new symbols. 
Yet to keep clear that the analysis is carried out in a local chart we
use the notation $L^2(\tO_j)$, $H^s(\tO_j)$ and not $L^2(\L)$, $H^s(\L)$.  
To further lighten notation we set $\tk_k = (\det g_k)^{1/2}\k_k$. One
has 
\begin{align*}
P_k = \d_t ^2 - (\tk_k)^{-1} \sum_{p, q} \d_p \tk_k g^{p q}_k \d_q +1 
  = \tilde{P}_k 
  - \sum_{p, q}   R^{p,q}_k, 
\end{align*} 
with $\tilde{P}_k = \d_t ^2 - \sum_{p,q} \d_p g^{pq}_k \d_q  + 1$ and $R^{p,q}_k = (\tk_k)^{-1} [\d_p, \tk_k] g^{pq}_k \d_q$. Note that
$\tchi_j B R_k^{p,q} \tchi_j$ defines a sequence of bounded operators
from $H^1(\L)$ into $L^2(\L)$, uniformly with
respect to $k$. Consequently, one has 
\begin{align*}
  \bigdup{\chi_j \tchi_j BR^{p,q}_k\tchi_j v^k_j}
  {\ovl{v_j^k}}_{H_{\comp}^{-1}(\tO_j), H_{\loc}^1 (\tO_j)}
  = \biginp{\chi_j \tchi_j BR^{p,q}_k\tchi_j v^k_j}
  {v_j^k}_{L^2 (\tO_j)}
  \ \ \mathop{\to}_{k\to + \infty} \ \
  0 
  \end{align*} 
 since $v _j^k$ converges strongly to $0$ in $L^2 (\tO_j)$.  This  leads to 
 \begin{align*}
   \bigdup{\chi_j \tchi_j BP_k \tchi_j v^k_j}{\ovl{v_j^k}}_{H_{\comp}^{-1}(\tO_j), H_{\loc}^1(\tO_j)}
   \ \ &= 
   \bigdup{\chi_j \tchi_j B \tilde{P}_k \tchi_j v^k_j}{\ovl{v_j^k}}_{H_{\comp}^{-1}(\tO_j), H_{\loc}^1(\tO_j)}+ o(1)_{k\rightarrow + \infty}\\
   &=
     \dup{\mu_j}{\chi_j  b p^0}_{S^* (\tO_j)}+ o(1)_{k\rightarrow + \infty},
 \end{align*} 
by  Proposition~\ref{prop: peu-reg-2}. Since $\chi_j \mu_j = \chi_j \mu$
locally, lifting back the analysis to the manifold level,  with \eqref{eq: localization
  measure property},
one finds
\begin{align*}
   \bigdup{BP_k  u^k}{\ovl{u^k}}_{H_{\comp}^{-1}(\L),
     H_{\loc}^1(\L)} 
   =
  \sum_{j \in J}\dup{\mu}{\chi_j  b p^0}_{S^* (\L)} 
  = \dup{\mu}{b p^0}_{S^* (\L)}+ o(1)_{k\rightarrow + \infty}. 
 \end{align*}
Now, one has 
\begin{align*}
   \bigdup{BP_k  u^k}{\ovl{u^k}}_{H_{\comp}^{-1}(\L),
     H_{\loc}^1(\L)} 
  =
  \bigdup{P_k  u^k}{\transp B \ovl{u^k}}_{H_{\loc}^{-1}(\L),
  H_{\comp}^1(\L)}+ o(1)_{k\rightarrow + \infty},
 \end{align*}
 with the transpose operator $\transp B$ bounded from $H_{\loc}^1(\L)$ into
 $H_{\comp}^1(\L)$ since $B$ is itself bounded from $H_{\loc}^{-1}(\L)$ into
 $H_{\comp}^{-1}(\L)$.
 If one assumes that $P_k u^k \to 0$
strongly in $H^{-1}_{\loc} \big((T_1,T_2) \times \M\big)$ one obtains
\begin{align*}
   \bigdup{BP_k  u^k}{\ovl{u^k}}_{H_{\comp}^{-1}(\L),
     H_{\loc}^1(\L)} 
   \ \ \mathop{\to}_{k\to + \infty} \ \ 0 ,
 \end{align*}
and thus
\begin{align*}
 \dup{\mu}{b p^0}_{S^* (\L)}= 0, \quad \forall b \in
 S^0_{c, \phg}(\L) \ \text{with} \  \supp( b ) \subset T^* \big(
 (T_1,T_2) \times \M\big),
\end{align*} 
and 
 one obtains the support estimation~\eqref{measure support}.

\bigskip We now prove the second item of the proposition. We assume
that $P_k u^k$ lies in $L^2_{\loc}\big((T_1,T_2)\times \M)$ and converges
strongly to $0$ in this space.  Consider $B \in \Psi^1_{c, \phg}(\L)$ with kernel 
supported in $\big( (T_1,T_2)\times \M\big)^2$ and 
$b\in S^1_{c, \phg}(\L)$ its principal symbol.  We are interested in
the limit of
$\bigdup{[P_k,B] u^k}{\ovl{u^k}}_{H_{\comp}^{-1}(\L),
  H_{\loc}^1(\L)}$, which makes sense since  $[P_k,B]$ is of order $2$. 
We have $ [P_k,B] u^k = P_k B u^k - B P_k u^k \in H^{-1}\big((T_1,T_2)\times \M)$. Since $P_k u^k $ lies
in $L^2\big((T_1,T_2)\times \M)$ by assumption, $B P_k u^k$ lies
in $H^{-1}\big((T_1,T_2)\times \M)$ and the same holds for $P_k B
u^k$. We may thus write 
\begin{align*}
 \bigdup{[P_k,B] u^k}{\ovl{u^k}}_{H_{\comp}^{-1}(\L), H_{\loc}^1(\L)}
  = \bigdup{P_kB u^k}{\ovl{u^k}}_{H_{\comp}^{-1}(\L),H_{\loc}^1(\L)} 
  - \bigdup{P_k u^k}{\ovl{B^* u^k}}_{L^2_{\loc}(\L), L^2_{\comp}(\L)},
\end{align*}
where the adjoint is computed with respect to the $L^2$-inner product
associated with $(k^0,g^0)$ here. 
As $B$ maps continuously $L_{\loc}^2 \big((T_1,T_2)\times \M)$ into $H_{\comp}^{-1}\big((T_1,T_2)\times \M)$, we have $B^*$ maps
continuously $H_{\loc}^1(\L)$ into $L_{\comp}^2(\L)$. Thus, one has 
\begin{align*}
  \biginp{P_k u^k}{B^* u^k}_{L^2(\L)}  
  \ \ \mathop{\to}_{k\to +\infty} \ \ 0.
\end{align*}
By Lemma~\ref{equivalence-densites} it is asymptotically equivalent to use $(\k^0,g^0)$
or $(\k_k, g_k)$ for the definition of the $L^2$-inner product and
$H^{-1}_{\comp}$-$H^1_{\loc}$ duality, that
is, 
\begin{align*}
 \bigdup{P_kB u^k}{\ovl{u^k}}_{H_{\comp}^{-1}(\L),H_{\loc}^1(\L)} 
  =
  \bigdup{P_kB u^k}{\ovl{u^k}}_{H_{\comp}^{-1}(\L, \k_k \mu_{g_k} dt),H_{\loc}^1(\L, \k_k \mu_{g_k} dt)}+ o(1)_{k\rightarrow + \infty}.
\end{align*}
Since $P_k$ is selfadjoint for this latter  $L^2$-inner product, one
obtains 
\begin{align*}
 \bigdup{P_kB u^k}{\ovl{u^k}}_{H_{\comp}^{-1}(\L),H_{\loc}^1(\L)} 
  &=
  \bigdup{B u^k}
        {\ovl{P_k u^k}}_{L^2_{\comp} (\L, \k_k \mu_{g_k} dt), L^2_{\loc} (\L, \k_k \mu_{g_k} dt)}+ o(1)_{k\rightarrow + \infty}\\
  &=
  \bigdup{B u^k}{\ovl{P_k u^k}}_{L^2_{\comp} (\L), L^2_{\loc} (\L)}+ o(1)_{k\rightarrow + \infty}.
\end{align*}
Using again that $P_k u^k \to 0$ 
strongly to $0$ in $L^2_{\loc}\big((T_1,T_2)\times \M)$ we obtain 
\begin{align*}
 \bigdup{P_kB u^k}{\ovl{u^k}}_{H_{\comp}^{-1}(\L),H_{\loc}^1(\L)} 
  \ \ \mathop{\to}_{k\to +\infty} \ \ 0,
\end{align*}
and finally
\begin{align}
  \label{eq: limit [Pk,b] uk uk}
 \bigdup{[P_k,B] u^k}{\ovl{u^k}}_{H_{\comp}^{-1}(\L), H_{\loc}^1(\L)}\
   \ \ \mathop{\to}_{k\to +\infty} \ \ 0.
\end{align}

As above, with the partition of unity $1= \sum_{j\in J}
\chi_j$ we write 
\begin{align}
  \label{eq: limit [Pk,b] uk uk-2}
  \bigdup{[P_k,B] u^k}{\ovl{u^k}}_{H_{\comp}^{-1}(\L), H_{\loc}^1(\L)} 
  &= \sum_{j \in J} \bigdup{\chi_j [P_k,B]u^k}
  {\ovl{u^k}}_{H_{\comp}^{-1}(\L), H_{\loc}^1(\L)} .
\end{align}
For each term in the sum one has 
\begin{align*}
  \bigdup{\chi_j [P_k,B]u^k}
  {\ovl{u^k}}_{H_{\comp}^{-1}(\L), H_{\loc}^1(\L)} 
  &=
    \bigdup{\chi_j [P_k, \tB_j] v^k_j}
  {\ovl{v^k_j}}_{H_{\comp}^{-1}(\L), H_{\loc}^1(\L)}+ o(1)_{k\rightarrow + \infty}.
\end{align*}
with $\tB_j = \tchi_j B \tchi_j$. 
This allows one to work in a local chart and write 
\begin{align}
  \label{eq: limit [Pk,b] uk uk-3}
   \bigdup{[P_k,B] u^k}{\ovl{u^k}}_{H_{\comp}^{-1}(\L),H_{\loc}^1(\L)}
  =  \sum_{j \in J}
  \bigdup{\chi_j [P_k,\tB_j] v^k_j}
  {\ovl{v^k_j}}_{H_{\comp}^{-1}(\tO_j), H_{\loc}^1(\tO_j)},
\end{align}
with the (manifold-local chart) identifications described above. 
With
$A_k = A_{\k_k, g_k}$, in the local chart $\mathcal C_j$ one writes
\begin{align*}
  \chi_j[P_k, \tB_j] = \chi_j[\d_t^2, \tB_j] - \chi_j[\A_k, \tB_j]
= \chi_j[\d_t^2, \tB_j] -\sum_{1\leq p,q\leq d} 
 \!\! \big( Q^{pq}_1 + Q^{pq}_2 
  + Q^{pq}_3 +Q^{pq}_4\big) ,
\end{align*} 
with 
\begin{align*}
  &Q^{pq}_1 = \chi_j\tk_k^{-1} \d_{x_p} \tk_k g_k^{p q} [\d_{x_q}, \tB_j], \quad 
  Q^{pq}_2 = \chi_j\tk_k^{-1} \d_{x_p} [ \tk_k  g_k^{p q}, \tB_j] \d_{x_q}, \\
  &Q^{pq}_3 = \chi_j\tk_k^{-1}[\d_{x_p}, \tB_j] \tk_k  g_k^{p q}\d_{x_q}, \quad 
  Q^{pq}_4 =  \chi_j [\tk_k^{-1}, \tB_j] \d_{x_p} \tk_k  g_k^{p q}\d_{x_q}.
\end{align*}

We now compute the limit of each term associated with this
decomposition of $[P_k, \tB_j]$ on the right-hand side of \eqref{eq: limit [Pk,b] uk uk-3}.
The principal symbol of $\chi_j [\d_t^2, \tB_j]$ is $i\chi_j \{ \tau^2, b\}$ and thus
\begin{align*}
  \bigdup{\chi_j [\d_t^2, \tB_j]  v^k_j}
  {\ovl{v^k_j}}_{H_{\comp}^{-1}(\tO_j), H_{\loc}^1(\tO_j)}
 =
  \dup{\mu_j}{ i \chi_j   \{ \tau^2, b\} }_{S^* (\tO_j)}+ o(1)_{k\rightarrow + \infty}.
\end{align*}
Proposition~\ref{prop: peu-reg-2}  applies and yields
\begin{align*}
  \bigdup{Q^{pq}_{1} v^k_j}
  {\ovl{v^k_j}}_{H_{\comp}^{-1}(\tO_j), H_{\loc}^1(\tO_j)}
 =
  \dup{\mu_j}{i \chi_j  g^{0, p q} \xi_p \d_{x_q} b}_{S^* (\tO_j)}+ o(1)_{k\rightarrow + \infty},
\end{align*}
 and 
\begin{align*}
  \bigdup{Q^{pq}_{3} v^k_j}
  {\ovl{v^k_j}}_{H_{\comp}^{-1}(\tO_j), H_{\loc}^1(\tO_j)}
=   \dup{\mu_j}{i \chi_j  g^{0, p q} (\d_{x_p} b) \xi_q }_{S^* (\tO_j)}+ o(1)_{k\rightarrow + \infty}.
\end{align*}

With Theorem~\ref{C-M} one has $[ \tk_k g_k^{p q}, \tB_j] \to [ \tk^0
g^{0,p q}, \tB_j]$ in $\mathscr L\big(L^2(\tO_j)\big)$ as $k \to +\infty$. It
follows that 
\begin{align*}
  \bigdup{Q^{pq}_{2} v^k_j}
  {\ovl{v^k_j}}_{H_{\comp}^{-1}(\tO_j), H_{\loc}^1(\tO_j)}
  =
  \bigdup{Q^{pq}_{2,a} v^k_j}
  {\ovl{v^k_j}}_{H_{\comp}^{-1}(\tO_j), H_{\loc}^1(\tO_j)}+ o(1)_{k\rightarrow + \infty},
\end{align*}
with 
\begin{align*}
  Q^{pq}_{2,a} = \chi_j\tk_k^{-1} \d_{x_p} [ \tk^0  g^{0, p q}, \tB_j] \d_{x_q}.
\end{align*}
With Corollary~\ref{cor: peu-reg} one writes 
\begin{align*}
  [ \tk^0  g^{0,p q}, \tB_j] = - \frac{1}{i} \nabla_{x}(\tk^0  g^{0,p q}) \cdot \Op
  \big(\nabla_{\xi} (\tchi_j^2 b)\big) + K_1,
\end{align*}
with $K_1$ a compact operator on $L^2(\R^{d+1})$, with compactly supported
kernel. One thus obtains
\begin{align*}
  \bigdup{Q^{pq}_{2} v^k_j}
  {\ovl{v^k_j}}_{H_{\comp}^{-1}(\tO_j), H_{\loc}^1(\tO_j)}
=
  \bigdup{Q^{pq}_{2,b} v^k_j}
  {\ovl{v^k_j}}_{H_{\comp}^{-1}(\tO_j), H_{\loc}^1(\tO_j)}+ o(1)_{k\rightarrow + \infty},
\end{align*}
with 
\begin{align*}
  Q^{pq}_{2,b} = - \frac{1}{i} \chi_j\tk_k^{-1} \d_{x_p} \nabla_{x}(\tk^0  g^{0,p q}) \cdot \Op
  \big(\nabla_{\xi} (\tchi_j^2 b)\big)\d_{x_q}.
\end{align*}
 Proposition~\ref{prop: peu-reg-2} applies and yields
\begin{align*}
  \bigdup{Q^{pq}_{2} v^k_j}
  {\ovl{v^k_j}}_{H_{\comp}^{-1}(\tO_j), H_{\loc}^1(\tO_j)}
  =
  \dup{\mu_j}{- i \chi_j  \xi_p \xi_q (\tk^0)^{-1}
  \nabla_{x}(\tk^0  g^{0,p q}) \cdot \nabla_{\xi} b }_{S^* (\tO_j)}+ o(1)_{k\rightarrow + \infty}.
\end{align*}

We now treat the term associated with $Q^{pq}_4$. 
Note that one has $\sum_{p,q}Q^{pq}_4 =  \chi_j [\tk_k^{-1}, \tB_j]
\tk_k A_k$. We write, lifting temporarily the analysis back to the manifold,
\begin{align*}
  \sum_{p,q}\bigdup{Q^{pq}_{4} v^k_j}
  {\ovl{v^k_j}}_{H_{\comp}^{-1}(\tO_j), H_{\loc}^1(\tO_j)}
  &= \bigdup{\chi_j [\tk_k^{-1}, B] \tk_k A_k v^k_j}
  {\ovl{v^k_j}}_{H_{\comp}^{-1}(\L), H_{\loc}^1(\L)}\\
    &=       \bigdup{\chi_j [\tk_k^{-1}, B] \tk_k A_k u^k}
      {\ovl{u^k}}_{H_{\comp}^{-1}(\L), H_{\loc}^1(\L)}+ o(1)_{k\rightarrow + \infty}.
\end{align*}
Setting $f^k = (\d_t^2 - A_k)u^k$ with $f^k \to 0$ strongly in
$L^2_{\loc}\big((T_1,T_2) \times \M\big)$, we thus find
 \begin{align*}
  \sum_{p,q}\bigdup{Q^{pq}_{4} v^k_j}
  {\ovl{v^k_j}}_{H_{\comp}^{-1}(\tO_j), H_{\loc}^1(\tO_j)}
  &=
      \bigdup{\chi_j [\tk_k^{-1}, B] \tk_k \d_t^2 u^k}
      {\ovl{u^k}}_{H_{\comp}^{-1}(\L), H_{\loc}^1(\L)}\\
    &\qquad  - \bigdup{\chi_j [\tk_k^{-1}, B] \tk_k f^k}
      {\ovl{u^k}}_{H_{\comp}^{-1}(\L), H_{\loc}^1(\L)}+ o(1)_{k\rightarrow + \infty}\\
   &=
      \bigdup{\chi_j [\tk_k^{-1}, \tB_j] \tk_k \d_t^2 v^k_j}
      {\ovl{v^k_j}}_{H_{\comp}^{-1}(\tO_j), H_{\loc}^1(\tO_j)}+ o(1)_{k\rightarrow + \infty},
 \end{align*}
 bringing again the analysis at the level of the local chart.
 
Using that $\tk_k$ is independent of $t$ we may write
\begin{align*}
  \chi_j [\tk_k^{-1}, \tB_j] \tk_k \d_t
  = \chi_j  \d_t [\tk_k^{-1}, \tB_j] \tk_k 
  + \chi_j [\tk_k^{-1}, E_j ]\tk_k,
 \end{align*}
where $E_j = [\d_t, \tB_j]\in \Psi^1_{c, \phg}(\tO_j)$, with $\d_t b\in
S^1_{c, \phg}(\tO_j)$ for principal symbol. 
With Theorem~\ref{C-M} we see that 
$[\tk_k^{-1}, E_j]$ maps $L^2(\tO_j)$ into itself continuously and
moreover $[\tk_k^{-1}, E_j] \to [(\tk^0)^{-1}, E_j]$ in $\mathscr
L\big(L^2(\tO_j)\big)$. Thus we obtain  
\begin{multline*}
  \bigdup{\chi_j [\tk_k^{-1}, E_j ]\tk_k \d_t v^k_j}
      {\ovl{v^k_j}}_{H_{\comp}^{-1}(\tO_j), H_{\loc}^1(\tO_j)}\\
=
  \bigdup{\chi_j [(\tk^0)^{-1}, E_j]\tk_k \d_t v^k_j}
      {\ovl{v^k_j}}_{H_{\comp}^{-1}(\tO_j), H_{\loc}^1(\tO_j)}+ o(1)_{k\rightarrow + \infty}
  \mathop{\to} _{k\to + \infty} \ \ 
  0,
\end{multline*}
arguing as above. 
Similarly we write 
\begin{align*}
  &\bigdup{\chi_j  \d_t [\tk_k^{-1}, \tB_j] \tk_k \d_t v^k_j}
      {\ovl{v^k_j}}_{H_{\comp}^{-1}(\tO_j), H_{\loc}^1(\tO_j)}\\
  &\qquad \qquad  \mathop{\sim} _{k\to + \infty} \ \ 
  \bigdup{\chi_j  \d_t [(\tk^0)^{-1}, \tB_j] \tk_k^0 \d_t v^k_j}
      {\ovl{v^k_j}}_{H_{\comp}^{-1}(\tO_j), H_{\loc}^1(\tO_j)}
  \end{align*}
Arguing as we did for the term associated with $Q_2^{p,q}$ we thus find
\begin{multline*}
  \bigdup{\chi_j  \d_t [\tk_k^{-1}, \tB_j] \tk_k \d_t v^k_j}
      {\ovl{v^k_j}}_{H_{\comp}^{-1}(\tO_j), H_{\loc}^1(\tO_j)}\\
 =
   \dup{\mu_j}{- i \chi_j  \tau^2 \tk^0  (\nabla_x (\tk^0)^{-1})\cdot
  \nabla_\xi b }_{S^* (\tO_j)}+ o(1)_{k\rightarrow + \infty}.
\end{multline*}

\medskip
Collecting the various estimates we found we obtain 
\begin{align}
  \bigdup{\chi_j [P_k,\tB_j] v^k_j}
  {\ovl{v^k_j}}_{H_{\comp}^{-1}(\tO_j), H_{\loc}^1(\tO_j)}
 = \dup{\mu_j}{\chi_j   \sigma}_{S^* (\tO_j)}+ o(1)_{k\rightarrow + \infty}.
\end{align}
with 
\begin{align*} 
  \sigma &= i \{ \tau^2, b\} 
           - i \sum_{p,q} \big( 
           g^{0, p q} \xi_p \d_{x_q} b 
           + g^{0, p q} (\d_{x_p} b) \xi_q 
           -  \xi_p \xi_q (\tk^0)^{-1} 
           \nabla_{x}(\tk^0  g^{0,p q})\cdot \nabla_{\xi} b \big)\\
  &\quad + i \tau^2 \tk^0 
    (\nabla_x (\tk^0)^{-1})\cdot \nabla_\xi b.
\end{align*}
Recalling that $p^0 = - \tau^2 +  \sum_{p,q} g^{0, p q} \xi_p \xi_q$
one finds
\begin{align*} 
  \sigma &=
           - i \{ p^0, b\} 
           + i p^0  (\tk^0)^{-1} \nabla_{x}(\tk^0)  \cdot \nabla_{\xi} b.
\end{align*}
Since $\mu$, and thus $\mu_j$, is supported in $\Char(p^0)$ by the
first part of the proposition, one
concludes that 
\begin{align*}
  \bigdup{\chi_j [P_k,\tB_j] v^k_j}
  {\ovl{v^k_j}}_{H_{\comp}^{-1}(\tO_j), H_{\loc}^1(\tO_j)}
 = -i \bigdup{\mu_j}{\chi_j   \{ p^0, b\} }_{S^* (\tO_j)}+ o(1)_{k\rightarrow + \infty}.
\end{align*}
Since $\chi_j \mu = \chi_j \mu_j$ (see Section~\ref{sec: Local representatives of the measure}), 
with \eqref{eq: limit [Pk,b] uk uk-2}--\eqref{eq: limit [Pk,b] uk
  uk-3} one obtains 
\begin{align*}
  \bigdup{[P_k,B] u^k}{\ovl{u^k}}_{H_{\comp}^{-1}(\L), H_{\loc}^1(\L)} 
= -i \bigdup{\mu}{\{ p^0, b\} }_{S^* (\L)}+ o(1)_{k\rightarrow + \infty}.
\end{align*}
With \eqref{eq: limit [Pk,b] uk uk}, this concludes the proof of  the
second part of  
the proposition since $\{p^0, b\} = \Hpref b$. 
 \end{proof}

\section{Measure support propagation: proof of Theorem ~\ref{th: ODE}}
\label{sec proof: th: ODE}

Theorem~\ref{th: ODE} is stated on an open subset of a smooth
manifold. Yet, its result is of a local nature. Using a
local chart we may assume that we consider an open of set $\Omega$ of $\R^d$
instead without any loss of generality.

The strategy we follow is very much inspired by the approach of
Melrose and Sj\"ostrand to the propagation of
singularities~\cite{MeSj78} and relies on careful choices of test
functions allowing one to construct sequences of points in the support
of the measure relying on nonnegativity\footnote{of the measure in our
case and of some operators for Melrose and Sj\"ostrand, via the
G\r{a}rding inequality.}. Then, a limiting procedure leads to the
conclusion, in the spirit of the classical proof of the Cauchy-Peano
theorem.

The proof of Theorem~\ref{th: ODE} is
made of two steps that are stated in the following propositions.
\begin{proposition}
\label{prop: invariant}
Let $X$ be a $\Con^0$-vector field on $\Omega$ an open set of
$\R^d$. For a closed set $F$ of $\Omega$, the following two
properties are equivalent.
\begin{enumerate}
\item \label{one} The set $F$ is a union of maximally extended integral curves of the vector field $X$.
\item \label{two} For any compact  $K\subset \Omega$ where the vector field $X$ does not vanish, 
\begin{align*}
  \forall \eps>0, \ \exists \delta_0>0, \  \forall x\in K\cap F,  
  \ \forall \delta \in [- \delta_0, \delta_0], 
  \ \ B\big(x+ \delta X(x), \delta \eps\big) \cap F \neq \emptyset.
\end{align*}
\end{enumerate}
\end{proposition}
\begin{proposition}
  \label{prop: symboles}
  Let $X$ be a $\Con^0$-vector field on $\Omega$ an open set of
  $\R^d$.  Consider a nonnegative measure $\mu$ on $\Omega$ that is a solution
  to  $\transp{X} \mu=0$
in the sense of distributions, that is, 
  \begin{equation}
  \label{EDO-bis}  
  \dup{\transp{X} \mu}{a}_{\D'(\Omega), \Cinf_c(\Omega)}
  = \dup{\mu}{X a}_{\D^{\prime,0}(\Omega), \Con^0_c(\Omega)}=0, 
  \qquad a \in \Cinfc(\Omega).
\end{equation}
Then, the closed set $F= \supp(\mu)$ satisfies the second property in
Proposition~\ref{prop: invariant}.
\end{proposition}

\medskip
\begin{proof}[Proof of Proposition~\ref{prop: invariant}]
 First, we prove that Property~\eqref{one} implies
  Property~\eqref{two} and consider a compact set $K$ of $\R^d$ such
  that $K \subset \Omega$ and $K \cap F \neq \emptyset$. 

  There exists $\eta>0$ such that $K \subset K_\eta \subset
  \Omega$ with $K_\eta = \{ x \in \Omega;\ \dist(x, K) \leq
  \eta\}$. One has $\Norm{X}{} \leq C_0$ on $K_\eta$ for some $C_0>0$. Let $x \in K$ and let $\gamma(s)$ be a maximal integral
  curve defined on an interval $]a,b[$, $a, b \in \ovl{\R}$ and such
  that $0 \in ]a,b[$ and $\gamma(0) =x$. If $b <
\infty$ then there exists $s^1 \! \! \in \, ]0,b[$ such that $\gamma(s^1)
\notin K_\eta$. Since $\gamma(s) \in K_\eta$ if $s <  \eta /C_0$, one finds
  that   $b \geq \eta /C_0$. Similarly, one has $|a| \geq \eta /C_0$.
  Consequently, there exists $\Slim>0$ such that any
  maximal integral curve $\gamma(s)$ of the vector field $X$ with
  $\gamma (0) \in K$ is defined for $s \in I= (- \Slim, \Slim)$. 
  
  Let us pick $x \in K$. According to
  Property~\eqref{one}, there exists
  \begin{align*}
    \gamma: I \rightarrow F \ \text{such that} \ 
    \dot{\gamma}(s) = X(\gamma(s)) 
     \ \text{and} \  \gamma(0) =x.
    \end{align*}
    By uniform continuity of the vector field $X $ in a compact
    neighborhood of $K$ we have
\begin{align*}
  \gamma(s) = \gamma(0) + \int_0^s \dot{\gamma}(s) ds
  = \gamma(0) + \int_0^s X(\gamma(s)) ds
  = x+ s X(x) + r(s),
  \qquad s\in (-S, S)
  \end{align*}
where $\lim_{s\rightarrow 0} \Norm{r(s)}{} / s =0$,
{\em uniformly} with respect to $x$. We deduce that for any  $\eps
>0$ there exists $0< \delta_0< S$ such that $\Norm{r(s)}{}< s \eps $  for any $s\in (- \delta_0, \delta_0)$,
which implies 
\begin{align*} 
  F \ni \gamma(s) \in B\big(x+ sX(x), s\eps\big).
\end{align*}

\medskip Second, we prove that Property~\eqref{two} implies
Property~\eqref{one}. It suffices to prove that for any $x\in F$ there
exist an interval $I \ni 0$ and an integral curve 
\begin{align*}  
  \gamma: I \rightarrow F 
  \ \text{such that} \ 
  \dot{\gamma}(s) =X( \gamma(s))
   \ \text{and} \ \gamma(0) =x. 
\end{align*}
Then, the standard continuation argument shows that this local
integral curve included in $F$ can be extended to a maximal integral
curve also included in $F$.

If $X(x)=0$, then the trivial integral curve
$\gamma(s) = x$, $s\in \R$, is included in $F$. As a consequence,
we assume $X(x) \neq 0$ and we pick a compact
neighborhood $K$ of $x$ containing  $B(x,\eta)$ with $\eta >0$ and where, for some $0 < c_K < C_K$, 
\begin{align*} 
 c_K\leq \Norm{ X(y)}{} \leq C_K, \quad y \in K.
\end{align*}

Let $n \in \N^*$. Set $x_{n,0} = x$ and  $\eps = 1 /n$ and apply
Property~\eqref{two}. One deduces that there exist $0< \delta_n \leq
1 / n$ and a point 
\begin{align*}  
x_{n,1}\in F \cap B \big(x_{n,0}+\delta_n X(x_{n,0}), {\delta_n} / n\big). 
\end{align*}
If $x_{n,1}\in K$ one can perform this construction again, starting
from $x_{n,1}$ instead of $x_{n,0}$. If a sequence of points
$x_{n,0}, x_{n,1}, \dots, x_{n,L^+}$ is obtained in this manner one has
\begin{align}  
  \label{eq: construction points - away from boundary}
x_{n,\ell+1} \in F \cap B\big( x_{n,\ell}+\delta_n X(x_{n, \ell}),
  {\delta_n} /n\big), \qquad \ell = 0, \dots, L^+-1.
\end{align}
\begin{figure}
  \begin{center}
    \subfigure[Iterative construction of the curve $\gamma_n$. \label{fig: construction gamma n-a}]
    {\resizebox{7cm}{!}{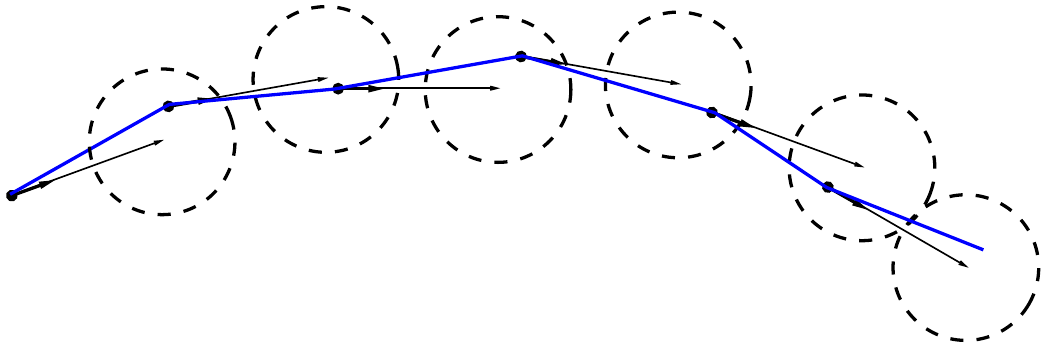}}
    \quad 
    \subfigure[Convergence of $\gamma_n$ as $n$ increases. \label{fig: construction gamma n-b}]
    {\resizebox{7cm}{!}{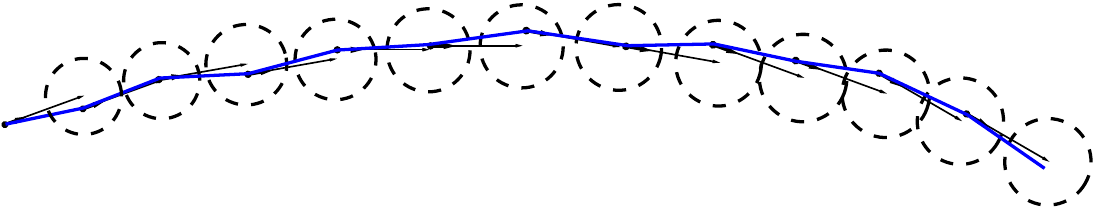}}
    \caption{Construction and convergence of the sequence $(\gamma_n)_n$.}
  \label{fig: construction gamma n}
  \end{center}
\end{figure}
One can
carry on the construction as long as $ x_{n,L^+} \in K$.
We can perform the same construction for $\ell\leq 0$, with the
property 
\begin{align}  
  \label{eq: construction points - away from boundary bis}
x_{n,\ell-1} \in F \cap B\big(x_{n,\ell} - \delta_n X(x_{n, \ell}),
  {\delta_n} / n\big), \qquad |\ell| = 0, \dots, L^--1.
\end{align}
Having $\Norm{X}{}\leq C_K$ on $K$ and $B(x,\eta) \subset K$ ensures that we can construct the sequence at least for 
\begin{align*}  
L^+ = L^- = L_n 
   = \Big\lfloor \frac{ \eta}{\delta_n (C_K + 1)}\Big\rfloor +1
  \leq \Big\lfloor \frac{ \eta}{\delta_n (C_K + 1/ n)} \Big\rfloor +1,
\end{align*}
where $\lfloor . \rfloor$ denotes the floor function. 
With the points $x_{n,\ell}$, $|\ell| \leq L_n$, we have constructed we
define the following continuous curve $\gamma_n(s)$ for
$|s|\leq L_n\delta_n$:
\begin{align*}  
\gamma_n (s) = x_{n,\ell}  + (s-\ell\delta_n) \frac{x_{n,\ell+1} - x_{n,\ell}} { \delta_n}
 \  \text{for}\  s\in [\ell \delta_n, (\ell + 1) \delta _n) \  \text{and}\  |\ell|
  \leq L_n-1.
\end{align*}
This curve and its construction is illustrated in Figure~\ref{fig:
  construction gamma n-a}. 
Note that $\gamma_n(s)$ remains in a compact set, uniformly with
respect to $n$. In this compact set $X$ is uniformly continuous. 

We set $\Slim = \eta /(C_K+1)$. Since $\Slim \leq L_n\delta_n$ we
shall in fact only consider the function $\gamma_n(s)$ for $|s| \leq
\Slim$ in what follows.  
Note that since $x_{n,\ell} \in F$ for $|\ell| \leq L_n$, one has 
\begin{align}
  \label{eq: discrete curve almost in F}
  \dist\big(\gamma_n (s), F\big)\leq  \delta_n ( C_K+ 1 / n), 
  \qquad |s|\leq \Slim.
\end{align}
From \eqref{eq: construction points - away from boundary}, for $\ell
\geq 0$ and  $s\in (\ell \delta_n, (\ell + 1) \delta _n)$, we have 
\begin{align*}  
  \dot{\gamma}_n(s)= \frac{x_{n,\ell+1} - x_{n,\ell}}{\delta_n} =
X(x_{n,\ell}) + \mathcal{O}( 1 / n).
\end{align*}
Similarly, from \eqref{eq: construction points -
  away from boundary bis}, for $\ell
\leq 0$ and  $s\in ((\ell-1) \delta_n, \ell\delta _n)$, we have 
\begin{align*}  
  \dot{\gamma}_n(s)= \frac{x_{n,\ell} - x_{n,\ell-1}}{\delta_n} =
X(x_{n,\ell}) + \mathcal{O}( 1 / n).
\end{align*}
In any case, using the uniform continuity of the
vector field $X$, we find 
\begin{align*}  
  \dot{\gamma}_n(s) 
  = X( \gamma_n(s)) + e_n(s),
\end{align*}
where the error $|e_n|$ goes to zero {\em uniformly} with respect to
$|s|\leq\Slim$ as $n\rightarrow + \infty$.

\medskip
Since the curve $\gamma_n$ is absolutely continuous (and differentiable except at isolated points), we find
\begin{equation}\label{eq.appro}
  \gamma_n(s) 
  = \gamma_n(0) + \int_0^s \dot{\gamma}_n(\sigma)d\sigma 
  = x+  \int_0^s X({\gamma}_n(\sigma)) d\sigma 
  + \int_0^s e_n(\sigma) \, d \sigma.
\end{equation}
We now let $n$ grow to infinity. With \eqref{eq.appro}, 
the family of curves
$(s \mapsto \gamma_n(s), |s| \leq\Slim)_{n\in \mathbb{N}^*}$ is
equicontinuous and pointwise bounded; by the Arzel\`a-Ascoli theorem we can extract a
subsequence $(s \mapsto \gamma_{n_p})_{p \in \mathbb{N}}$ that
converges uniformly to a curve $\gamma(s), |s| \leq \Slim$. Convergence is illustrated in Figure~\ref{fig:
  construction gamma n-b}.  Passing
to the limit $n_p\rightarrow+ \infty$ in~\eqref{eq.appro} we find that
$\gamma(s)$ is solution to
\begin{align*}  
\gamma(s) = x+  \int_0^s X({\gamma}(\sigma))d\sigma.
\end{align*}
From estimation~\eqref{eq: discrete curve almost in F}, for any $|s|
\leq \Slim$, there exists $(y_p)_p \subset F$ such that 
$\lim_{p\rightarrow + \infty} y_p = \gamma(s)$.
Since $F$ is closed we conclude that $\gamma(s) \in F$. 
\end{proof}
  
\medskip
\begin{proof}[Positivity  argument and proof of Proposition~\ref{prop:
    symboles}]
 We
  consider a compact set $K$ where the vector field $X$ does not
  vanish.  By continuity of the vector field there exist $0< c_K\leq  C_K$ such that  $0< c_K \leq\Norm{ X(x)}{}
  \leq C_K$, for all $x \in K$. 

Let us consider $x^0\in K\cap \text{supp}(\mu)$. By performing a
rotation and a dilation by a factor $\Norm{ X(x^0)}{}$, we
can assume that $X(x^0) = (1,0, \dots, 0)\in \R^d$. We shall write $x
= (x_1, x')$ with $x' \in \R^{d-1}$. 

\medskip Let $\chi \in \Cinf(\R)$ be given by 
\begin{equation}
  \label{eq: def chi}
  \chi (s) = \bld{1}_{s<1}\, \exp (1/(s-1)),
\end{equation} 
and 
$\beta \in \Cinf(\R)$ be such that
\begin{equation}
  \label{eq: def beta}
  \beta \equiv 0 \ {\text on}\   ]-\infty, -1],\quad 
 \beta'>0 \ {\text{on}}\  ]-1, -1/2[,\quad
 \beta \equiv 1\ {\text{on}}\  [-1/2, +\infty[.
\end{equation} 
We then set 
\begin{align} 
  \label{eq: def q eps delta x - test function}
  q_{\eps, \delta,x^0} = (\chi \circ v) (\beta \circ w), 
  \quad g_{\eps,\delta,x^0} =(\chi' \circ v) (\beta \circ w)  X v, 
  \quad h_{\eps,\delta,x^0} = (\chi \circ v) (\beta' \circ w) X w, 
  \end{align}
with 
\begin{align*}
  v(x)= 1/2-\delta^{-1} (x_{1}- x^0_{1})
  +8 (\eps \delta)^{-2} \Norm{x'- x^{0\prime}}{}^{2}
  \quad \text{and} \quad 
  w(x)=2\eps^{-1} \big(1  - \delta^{-1} (x_{1}-x^0_{1}) \big),  
\end{align*}
for $\eps>0$ and $\delta>0$ both meant to be chosen small in what follows.
We have $X q_{\eps, \delta,x^0}  =  g_{\eps,\delta,x^0} +
h_{\eps,\delta,x^0}$. 

\medskip
The function $q_{\eps, \delta,x^0} $ is compactly supported. Indeed, in
the support of $\beta\circ w$ one has $w\geq -1$ implying
\begin{align*}
  x_1- x^0_1 \leq \delta (1+ \eps/2),
\end{align*}
while on the support of $\chi \circ v$ one  has $v \leq 1$ which gives
\begin{equation*}
  - 1/2+8  (\eps \delta)^{-2} \Norm{x'- x^{0\prime}}{}^{2}
  \leq \delta^{-1} (x_{1}- x^0_{1}).
\end{equation*}
On the supports of $q_{\eps,\delta,x^0}$ and $(\chi'
\circ v) (\beta \circ w)$ one thus finds
\begin{equation}
  \label{proche}
  - \delta /2  \leq  x_1- x^0_1 \leq \delta ( 1+ \eps/2)
  \ \  \text{and} \ \ 
  8  (\eps \delta)^{-2} \Norm{x'- x^{0\prime}}{}^{2} \leq 3/2 + \eps/2.
\end{equation}
Similarly, on the support of $\beta' \circ w$ one has $-1\leq w \leq
-1/2$ yielding
\begin{align*}
  \delta ( 1+ \eps /4) \leq x_1- x^0_1 \leq \delta (1+ \eps/2),
\end{align*} 
which implies that on the support of
$h_{\eps,\delta,x^0}$ one has
 \begin{equation}\label{proche2}
   \delta ( 1+ \eps /4)  \leq x_1- x^0_1 
   \leq \delta ( 1+ \eps/2)
   \ \  \text{and} \ \ 
   8  (\eps \delta)^{-2} \Norm{x'- x^{0\prime}}{}^{2} \leq 3/2 + \eps/2.
\end{equation}
In particular, in the case $\eps \leq 1$, one finds
 \begin{equation}\label{proche3}
   \supp(h_{\eps,\delta,x^0}) 
   \subset B\big(x^0 + \delta X(x^0), \eps \delta\big).
\end{equation}
These estimations of the supports of $q_{\eps,\delta,x^0}$ and
$h_{\eps,\delta,x^0}$ are illustrated in Figure~\ref{fig: test function supports}.
\begin{figure}
  \begin{center}
    \subfigure[Support of $q_{\eps,\delta,x^0}$. \label{fig: q eps delta support}]
    {\resizebox{7.5cm}{!}{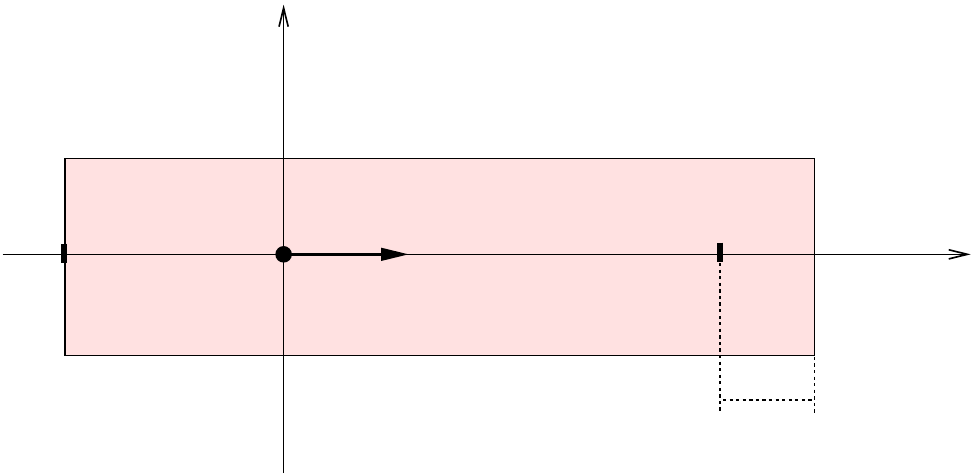}}
    \quad 
    \subfigure[Support of $h_{\eps,\delta,x^0}$. \label{fig: g eps delta support}]
    {\resizebox{7.5cm}{!}{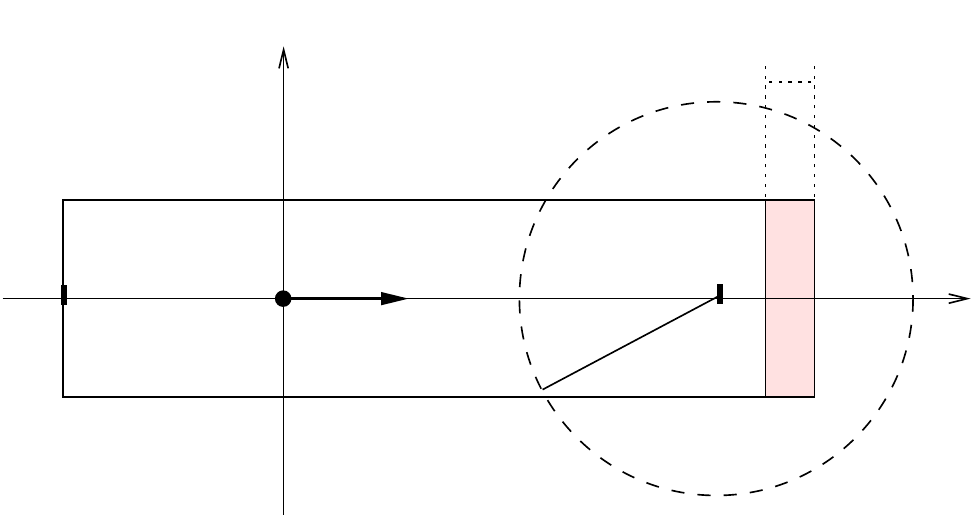}}
    \caption{Estimation of the test function supports in the case $\eps \leq 1$.}
  \label{fig: test function supports}
  \end{center}
\end{figure}
\begin{lemma}
  \label{lem: positivite g}
  For any $0 < \eps \leq 1$ there exists $\delta_0>0$ such that for
  any $x^0\in K$ and $0 < \delta \leq \delta_0$, the function
  $g_{\eps, \delta,x^0}$ is nonnegative.
  Moreover, $g_{\eps, \delta,x^0}$ is positive in a \nhd
  of $x^0$. 
\end{lemma}
\begin{proof}
  Let $0 < \eps \leq 1$.  We have
  $g_{\eps, \delta,x^0} = (\chi' \circ v) (\beta \circ w) X v$. Since
  $\beta \geq 0$ and $\chi'<0$ it suffices to prove that $X v(x) \leq 0$
  for $x$ in the support of $(\chi' \circ v) (\beta \circ w)$ for $\delta>0$
  chosen sufficiently small, uniformly with respect to $x^0 \in K$.

  We write 
  \begin{align*}
    X(x) - X(x^0) = \alpha^1(x, x^0) \d_{x_1} + \alpha'(x, x^0) \cdot \nabla_{x'},
  \end{align*}
  with $\alpha^1(x, x^0)\in \R$ and $\alpha'(x, x^0)\in \R^{d-1}$. 
  By \eqref{proche},  for $x \in \supp (\chi' \circ v) (\beta \circ w)$ we
  have $\Norm{ x - x^0}{} \lesssim \delta$. From the uniform continuity of
  $X$ in any compact set we conclude that 
  \begin{align}
    \label{eq: smallness Xx - Xx0}
    |\alpha^1(x, x^0)| + \Norm{\alpha'(x, x^0)}{} = o(1) \ \
    \text{as} \  \delta \to 0^+,
  \end{align}
  uniformly\footnote{Observe that the change of variables made above
    for $X(x^0) = (1,0,\dots,0)$ does not affect uniformity since the
    dilation is made by a factor in $[c_K,C_K]$.} with respect to $x^0\in K$ and $x \in \supp (\chi' \circ v) (\beta \circ w)$.
  Using that $X(x^0) = \d_{x_1}$ and the form of $v$ given above, we write 
  \begin{align*}
    X v(x) &= \big(X(x) v\big)(x) 
    = \big(\d_{x_1} v + \big(X(x) - X(x^0)\big) v\big)(x) \\
    &= - \delta^{-1}\Big( 1 + \alpha^1(x, x^0)  - 16 \eps^{-1}(\eps \delta)^{-1} \alpha'(x,x^0) \cdot (x' -
      x^{0\prime})\Big).
  \end{align*}
  Using again \eqref{proche}, we thus find for $x \in \supp (\chi'
  \circ v) (\beta \circ w)$  
  \begin{align*}
    \big| \alpha^1(x, x^0) 
    -  16 \eps^{-1}(\eps \delta)^{-1} \alpha'(x,x^0) \cdot (x' - x^{0\prime})
    \big| 
    \lesssim \big| \alpha^1(x, x^0) \big| + \eps^{-1} \Norm{\alpha'(x,x^0)}{}.
  \end{align*}
  With $\eps$ fixed above and with \eqref{eq: smallness Xx - Xx0} we
  find that $X v(x) \sim - \delta^{-1}$ as $\delta \to 0^+$
  uniformly with respect to $x^0\in K$ and
  $x \in \supp (\chi' \circ v) (\beta \circ w)$.

  Finally, we have 
  $g_{\eps, \delta,x^0}(x^0) = - \delta^{-1}\chi'(1/2) \beta(2
  \eps^{-1})>0$ and thus $g_{\eps, \delta,x^0}$ is positive in a \nhd
  of $x^0$. 
\end{proof}

We are now in a position to conclude the proof of
Proposition~\ref{prop: symboles}. Note that it suffices to prove the result
for $0 < \eps \leq 1$.  We choose $\delta_0>0$ as given by
Lemma~\ref{lem: positivite g}.
Let then $x^0 \in K \cap \supp(\mu)$. We apply~\eqref{EDO-bis} to the  family  $q_{\eps, \delta, x^0}$ of test
 functions with $0 < \delta \leq \delta_0$:
\begin{equation}\label{somme}
  0= \langle \mu, X(q_{ \eps, \delta, x^0})\rangle 
  = \langle \mu, g_{ \eps, \delta, x^0}\rangle
  +  \langle \mu, h_{ \eps, \delta, x^0}\rangle.
\end{equation}
By Lemma~\ref{lem: positivite g},  $g_{\eps, \delta, x^0}\geq 0$ and $g_{\eps,
  \delta,x^0}$ is positive in a \nhd of $x^0$. As $x^0 \in \supp(\mu)$
we find   $\langle \mu, g_{\eps, \delta, x^0}\rangle>0$. Consequently, 
$\langle \mu, h_{ \eps, \delta, x^0}\rangle\neq 0$. 
By the support estimate for $h_{ \eps, \delta, x^0}$ given in \eqref{proche3} the conclusion follows: $\supp(\mu) \cap
B\big(x^0 + \delta X(x^0), \eps \delta\big)\neq \emptyset$. 
\end{proof}

\section{Exact controllability: proof of Theorem~\ref{theoprinc}}
\label{prooftheoprinc}

Let $ (\k^0,g^0) \in \X^1(\M)$ and assume that $(\omega, T)$ fulfills the
geometric control condition of Definition~\ref{def: Geometric control condition-ter}.

Let also $(\k, g) \in \Y(\M)$. With Proposition~\ref{prop: equiv
  controllability observability}, the result of Theorem
~\ref{theoprinc} follows if we prove that there exists $\eps>0$ and $\Cobs>0$ such
that 
\begin{equation*}
  \E_{\k, g} (u) (0)
    \leq \Cobs
    \Norm{\bld{1}_{(0,T) \times \omega}\, \d_t u}{L^2(\L, \k \mu_g dt)}^2,
\end{equation*}
for any weak solution $u$ of the wave equation associated with
$(\k, g)$ chosen such that
\begin{align*}
  \Norm{(\k, g) - (\k^0,g^0) }{\Y(\M)}\leq \eps.
\end{align*}
The
$L^2$-norm on the \rhs is associated with $(\k,g)$,
that is,
\begin{equation*}
  \Norm{\bld{1}_{(0,T) \times \omega}\, \d_t u}{L^2(\L, \k \mu_g dt)}^2
  = \int_0^T \int_\omega |\d_t u|^2\ \k \mu_g d t.
\end{equation*}
Yet, for $\eps>0$ chosen \suff small one has 
$\Norm{.}{L^2(\L, \k^0 \mu_{g^0})}
  \eqsim
  \Norm{.}{L^2(\L, \k \mu_g)}$,
where $ A \eqsim B$ means $c_1 \leq A/B \leq c_2$ for some $c_1, c_2 >0$. In other words, we have equivalence with constants uniform with respect to $(\k, g)$.
In what follows,  $L^2$- and more generally $H^s$-norms on $\M$ are
chosen with respect to $\k^0 \mu_{g^0}$ unless explicitly written.
Our goal is thus to prove the  observability inequality
\begin{equation}
  \label{equ: comp-obs}
  \E_{\k^0, g^0} (u) (0)
  \leq \Cobs
  \Norm{\bld{1}_{(0,T) \times \omega}\, \d_t u}{L^2(\L)}^2.
\end{equation}
The Bardos-Lebeau-Rauch uniqueness compactness argument reduces the proof
of \eqref{equ: comp-obs} to the proof of
the weaker estimate
\begin{equation}
  \label{equ: compact obs}
  \E_{\k^0, g^0}(u)(0)
  \leq C   \Norm{\bld{1}_{(0,T) \times \omega}\, \d_t u}{L^2(\L)}^2
  + C' \bigNorm{\big(u(0), \d_tu(0)\big)}{L^2(\M) \oplus H^{-1}(\M)}^2,
\end{equation}
which exhibits an additional compact term, and expresses observability
for high-frequencies. Low frequencies are dealt with by means of a
unique continuation argument.

To prove \eqref{equ: compact obs} we argue by contradiction and we
assume that there exists a sequence
$(\k_k,g_k)_{k \in \N} \subset \Y(\M)$ such that
\begin{equation}\label{equ: metriclimit}
\lim_{k\rightarrow + \infty } \Norm{(\k_k, g_k) - (\k^0,g^0)}{\Y(\M)}= 0,
\end{equation}
and yet for each $k \in \N$ the associated observability inequality
does not hold. Thus, 
for each $k \in \N$, there exists a sequence
of initial data  $\big(v^{k,p, 0}, v^{k,p, 1}\big)_{p \in \N} \subset H^1(\M)
\times L^2(\M)$ with associated  solution $(v^{k,p})_{p \in \N}$,
that is,
\begin{align*}
  \begin{cases}
     P_{k} v^{k,p}  =0 
    & \text{in}\ (0,+\infty)\times\M,\\
      {v^{k,p}}_{|t=0}  = v^{k,p, 0}, \ \d_t {v^{k,p}}_{|t=0}  = v^{k,p, 1}
      & \text{in}\ \M,
  \end{cases}
\end{align*}
with $P_k = P_{\k_k, g_k}$, that moreover has the properties
\begin{equation*}
  \E_{\k^0, g^0}(v^{k,p})(0) =1
  \ \ \text{and} \ \
 \Norm{\bld{1}_{(0,T) \times \omega}\, \d_t v^{k,p}}{L^2(\L)}
+ \bigNorm{\big(v^{k,p, 0}, v^{k,p, 1}\big)}{L^2(\M) \oplus H^{-1}(\M)}
\leq \frac{1}{p+1} .
\end{equation*}
We take $p=k$ and we set$\big(u^{k,0}, u^{k,1} \big)= \big(v^{k,k, 0}, v^{k,k, 1}\big)$ and $u^k=v^{k,k}$, one obtains
$P_k u^k =0$   in $\L$ 
and
\begin{equation}\label{mass+obslimit}
  \E_{\k^0, g^0}(u^{k})(0) =1
  \ \ \text{and} \ \
  \Norm{\bld{1}_{(0,T) \times \omega}\, \d_t u^ k}{L^2(\L)}
+ \bigNorm{\big(u^{k,0}, u^{k,1} \big)}{L^2(\M) \oplus H^{-1}(\M)}
\leq \frac{1}{k+1}.
\end{equation}
From~\eqref{mass+obslimit} one has $u^k \rightharpoonup 0$ weakly  in $H^1_{\loc}(\L)$. 
With \eqref{eq: exists mdm mu 0}--\eqref{eq:  exists mdm mu 1}, we can
associate with (a subsequence of) $(u^k)_k$ an $H^1$-microlocal defect
measure $\mu $ on $S^{*} (\L)$. Here, the measure is understood with
respect to $L^2(\L, \k^0 \mu_{g^0} dt)$.
\medskip
From the second part of \eqref{mass+obslimit} one has
\begin{align}
\label{eq: mu vanishes on (0,T)xomega}
  \mu =0\  \text{in} \ S^* ( (0,T) \times \omega).
\end{align}
In fact, for any $\psi \in \Cinf((0,T) \times \omega)$ one has
$\Norm{\psi \d_t u^ k}{L^2(\L)} \sim 0$ and thus
$\dup{\mu}{\tau^2 \psi^2}=0$. Hence,
$\supp(\mu) \cap S^*((0,T) \times \omega) \subset \{ \tau =0\}$. Since
$\{ \tau =0\} \cap \Char(p^0)\cap S^*(\L) = \emptyset$ with \eqref{measure support} one
  obtains \eqref{eq: mu vanishes on (0,T)xomega}.

With the first part of \eqref{mass+obslimit} one has the
following lemma.
\begin{lemma}
  \label{lemma: non vanishing measure}
  The measure $\mu$ does not vanish on $S^*\big( \L \big)$.
\end{lemma}
A proof is given below.

\medskip
We now use Proposition~\ref{prop-int} to obtain a precise description of the
measure $\mu$. First, one has $\supp(\mu) \cap S^*  ((0, T)\times \M)\subset \Char(p^0)$.
Furthermore, one has $\transp{\Hpref} \mu =0$ in the sense of
distributions on $S^*  ((0, T)\times \M)$. Since $\Hpref$ is a $\Con^0$-vector field
on the manifold $S^* \L$, Theorem~\ref{th: ODE} implies that 
$\supp(\mu)$ is a union of maximally
extended \bichars in $S^*  ((0, T)\times \M)$.

Under the geometric control condition of
Definition~\ref{def: Geometric control condition-ter}, any
maximal \bichar meets $S^*((0, T) \times \omega)$ where $\mu$ vanishes
by \eqref{eq: mu vanishes on (0,T)xomega}. Thus $\supp(\mu) =
\emptyset$ yielding a contradiction with the result of
Lemma~\ref{lemma: non vanishing measure}. We thus 
obtain that \eqref{equ: comp-obs}
holds. This concludes the proof of Theorem~\ref{theoprinc}.
\hfill \qedsymbol \endproof

\begin{proof}[Proof of Lemma~\ref{lemma: non vanishing measure}]
  Let  $T_1 < T_2$ and $\phi \in \Cinfc (\R)$ nonnegative and equal to
  $1$ on a \nhd of $[T_1,T_2]$. On $\L$, consider the {\em elliptic}
  operator $Q = -\d^2_t - A_{\k^0,g^0} + 1$ with symbol
  $q = \tau^2 + \sum_{p,q} g^{0\, p,q} (x) \xi_p \xi_q$.  Taking
  \eqref{eq: exists mdm mu 1} and Lemma \ref{equivalence-densites}
  into account one can write
  \begin{align}\label{mu-lim}
    \dup{ \phi^2 Q u^k}{\ovl{u^k}}_{H_{\comp}^{-1}(\L), H_{\loc}^1(\L)}  
    \ \ \mathop{\sim}_{k\to + \infty} \ \
      \dup{\mu}{\phi^2 q}_{S^*\L}.
  \end{align}
  Integrating by parts one obtains
 \begin{align*}
  &\dup{ \phi^2 Q u^k}{\ovl{u^k}}_{H_{\comp}^{-1}(\L), H_{\loc}^1(\L)}  \\
   &\qquad \quad
     = \int_\L \phi (t)^2\Big(|\d_tu^k|^2
     + g^0 (\nabla_{\! g^0}u^k,  \nabla_{\! g^0}\overline{u^k}) + |u^k|^2 \Big)
     \k^0 \mu_{g^0} d t
     +  2 \inp{\phi' \phi\ \d_tu^k}{u^k}_{L^2(\L)}\\
   &\qquad \quad
     =\int_\R \phi(t)^2 \E_{\k^0, g^0}(u^k)(t) dt
     +  2 \inp{\phi' \phi \ \d_tu^k}{u^k}_{L^2(\L)}.
 \end{align*}
 Since the energy built on $\k^p, g^p$ is preserved by the evolution given by $P_p$, we have  by \eqref{mass+obslimit}
 \begin{multline}
 \E_{\k^0, g^0}(u^k)(t) = \E_{\k^k, g^k}(u^k)(t)+ o(1)= \E_{\k^k, g^k}(u^k)(0)+ o(1)= \E_{\k^0, g^0}(u^k)(0) + o(1)\\
 =1+ o(1)\end{multline}
 and since $\inp{\phi' \phi \ \d_tu^k}{u^k}_{L^2(\L)}\to 0$ as $u^k
 \to 0$ strongly in $L^2_{\loc}(\L)$, one obtains
   \begin{align*}
    \dup{ \phi^2 Q u^k}{\ovl{u^k}}_{H_{\comp}^{-1}(\L), H_{\loc}^1(\L)} 
   \ \ \mathop{\sim}_{k\to + \infty} \Norm{\phi}{L^2(\R)}^2.
\end{align*}
With \eqref{mu-lim} this proves that $\mu \neq 0$.
\end{proof}

\section{Lack of continuity of the control operator with respect to 
  coefficients}
\label{sec: lack continuity HUM}
\subsection{Proof of Theorems~\ref{HUM operator} and \ref{HUM operator bis}}
\label{sec: proof HUM operator}

We prove the result of both theorems, that is, in the case $k\geq
1$. In the case $k=1$ we are simply required to prove additionally  that the
geometric control condition of Definition~\ref{def: Geometric
  control condition-thm} is fulfilled for geodesics given by the
chosen metric $\tg$; see Remark~\ref{rem: HUM operator - GCC}.

Let $\eps>0$. We set $\tg = (1+\eps)g$. Given
any neighborhood $\mathcal{U}$ of $( \k , g)$ in $\X^k(\M)$, for
$\eps>0$ chosen \suff small one has $(\k, \tg) \in \mathcal{U}$.

Moreover, observe that for $\eps>0$ chosen \suff small geodesics
associated with $\tg$ can be made arbitrarily close to those associated
with $g$ uniformly in $t \in [0,T]$. Hence, for such $\eps>0$ 
the geometric control condition is fulfilled for geodesics associated
with $\tg$.

Observe that one has 
\begin{align}
  \label{eq: non inersection characteristic sets}
  \Char(p_{\kappa, g}) \cap  \Char(p_{\kappa,\tg})\cap S^* \L = \emptyset. 
\end{align}

We consider a sequence $(y^{k,0},y^{k,1}) \rightharpoonup (0,0)$
weakly in $H^1 (\M)\oplus L^2(\M)$ such that
 \begin{align*}
   \frac12 (\Norm{y^{k,0}}{H^1(\M)}^2
   + \Norm{y^{k,1}}{L^2(\M)}^2 )=1.
 \end{align*}
$L^2$- and $H^1$-norms are based on the $\k \mu_g dt$ measure on $\L$.
 
 Setting
 $f_{\k, g}^k = H_{\k,g}(y^{k,0},y^{k,1}) \in L^2((0,T)\times \M)$
 with $H_{\k, g}$ defined in \eqref{eq: def HUM operator}, 
 one obtains a sequence of control functions.
According to the HUM
 method \cite{Lions}, $f_{\kappa, g}^k$  is itself a (weak) solution to the
 following free wave equation
 \begin{equation}
   \label{eq-control}
   P_{\kappa, g} f_{\kappa, g}^k =0,
 \end{equation}
 in the energy space $L^2(\M)\oplus H^{-1}(\M)$, that is,
 $(f_{\kappa, g}^k(0), \d_tf_{\kappa, g}^k(0)) \in L^2(\M)\times
 H^{-1}(\M)$. Moreover, $(f_{\kappa, g}^k(0), \d_tf_{\kappa, g}^k(0))$
 depend continuously on $(y^{k,0},y^{k,1})$. The function $f_{\kappa,
   g}^k$ is thus bounded in $\Con^0((T_1, T_2), L^2(\M))$ uniformly
 with respect to $k$ for any $T_1 < T_2$. 
 Since the map $H_{\k,g}$ is continuous $f_{\kappa, g}^k
 \rightharpoonup 0$   weakly  in $L^2_{\loc}(\L)$. Up to
extraction of a subsequence, it is associated with an $L^2$-microlocal
defect measure $\mu_f$. With Proposition~\ref{prop-int-bis} one has
\begin{equation}
  \label{eq: location supp mu f}
  \supp(\mu_f) \subset \Char(p_{\k,g}).
 \end{equation}

 We consider the sequences of solutions $(y^k)_k$ and $(\ty^k)_k$ to 
\begin{equation*}
  \begin{cases}
    P_{\k, g} y^k
    =  \bld{1}_{(0,T) \times \omega} \,f_{\kappa, g}^k & \text{in} \  \L,\\
    (y^k,\d _{t}y^k)_{|t=0}=(y^{k,0},y^{k,1})& \text{in} \ \M,
  \end{cases}
  \qquad 
  \begin{cases}
    P_{\k, \tg} \ty^k
    =  \bld{1}_{(0,T) \times \omega} \, f_{\kappa, g}^k  & \text{in} \  \L,\\
    (\ty^k,\d _{t}\ty^k )_{|t=0}=(y^{k,0},y^{k,1})& \text{in} \ \M.
  \end{cases}
\end{equation*}
Both are bounded and weakly converge to $0$ in $H^1_{\loc}(\L)$. Up
to extraction of subsequences, both are associated with
$H^1$-microlocal defect density measures $\mu$ and $\tmu$
respectively.  Since
$\bld{1}_{(0,T) \times \omega} f_{\kappa, g}^k \rightharpoonup 0$
weakly in $L^2_{\loc}(\L)$, we have
$\bld{1}_{(0,T) \times \omega} f_{\kappa, g}^k\rightarrow 0$ strongly
in $H^{-1}_{\loc}(\L)$ and, with Proposition~\ref{prop-int}, one finds
$\supp(\tmu) \subset \Char(p_{\k,\tg})$.
Thus one has 
\begin{align}
  \label{eq: support t mu - pertub result}
  \supp(\tmu) \cap \supp(\mu_f) = \emptyset.
\end{align}
The sequence $(\d_t \ty^k)$ converges to 0 weakly in $L^2_{\loc}(\L)$ and can
be associated with an $L^2$-microlocal
defect density measure whose support is given by $\supp(\tmu)$. 
\begin{lemma}
  \label{lemma: mesure croisee}
  One has $\inp{\bold{1}_{(0,T)\times \omega}\ f^k_{\kappa, g}}
  {\d_t \ty^k}_{L^2(\L, \k \mu_{\tg} dt)} \to 0$ as $k \to +\infty$.
\end{lemma}
A proof is given below.

  Using the density of strong solutions of the wave equation, with
  integration by parts, one finds the  classical energy estimate
  \begin{equation*}
    \mathcal{E}_{\k, \tg}(\ty^k)(T)- \mathcal{E}_{\k, \tg}(\ty^k)(0)
    = \inp{\bold{1}_{(0,T)\times \omega}\ f^k_{\kappa, g}}
      {\d_t \ty^k}_{L^2(\k \mu_{\tg} dt)}.
    \end{equation*}
With Lemma~\ref{lemma: mesure croisee} one obtains
\begin{align*}
  \mathcal{E}_{\k, \tg}(\ty^k)(T) \ \  \mathop{\sim}_{k \to +\infty} \ \
  \mathcal{E}_{\k, \tg}(\ty^k)(0).
\end{align*}
With the form of $\tg$ chosen above one has
\begin{equation*}
  \mathcal{E}_{\k, g}(\ty^k)(t) = (1+\mathcal{O}(\eps) )
  \mathcal{E}_{\k, \tg}(\ty^k)(t), 
\end{equation*}
uniformly with respect to $t \in [0,T]$. Choosing $\eps>0$ \suff small
and $k$ \suff large, 
the first part of Theorem~\ref{HUM operator}  follows since $\mathcal{E}_{\k, g}(\ty^k)(0)=1$.

\medskip
We use the values of $\eps$ and $k$ chosen above. 
To prove \eqref{HUM-noncontinu}, we write $\ty^k$ in the form
$\ty^k=v_1 + v_2$  where
$v_1$ and $v_2$ are solutions to
\begin{equation}
  \label{eq: equa-difference}
  \begin{cases}
    P_{\k, \tg} v_1 
    =  \bld{1}_{(0,T) \times \omega} \,f_{\kappa, \tg}^k & \text{in} \  \L,\\
    (v_1,\d _{t}v_1)_{|t=0}=(y^{k,0},y^{k,1})& \text{in} \ \M,
  \end{cases}
  \qquad 
  \begin{cases}
    P_{\k, \tg} v_2
    =  \bld{1}_{(0,T) \times \omega}  ( f_{\kappa, g}^k - f_{\kappa, \tg}^k) & \text{in} \  \L,\\
    (v_2,\d _{t}v_2)_{|t=0}=(0,0)& \text{in} \ \M.
  \end{cases}
\end{equation}
with $f_{\kappa, \tg}^k = H_{\k, \tg} (y^{k,0},y^{k,1})$.
A  hyperbolic energy estimation for the solution $v_2$  to the second
equation in \eqref{eq: equa-difference} gives
\begin{align*}
  \mathcal{E}_{\k, \tg}(v_2)(T)
  \leq C_T \Norm{\bld{1}_{(0,T) \times \omega} (f_{\kappa, g}^k -
  f_{\kappa, \tg}^k)}{L^2 (\L)}^2.
\end{align*}
Since one has $(v_1(T),\d _{t}v_1(T))=(0,0)$ because of the definition of
$f_{\kappa, \tg}^k$ one finds
\begin{align*}
  \mathcal{E}_{\k, \tg}(v_2)(T) = \mathcal{E}_{\k,\tg}(\ty^k)(T)
  \geq 1/2,
\end{align*}
which gives the
second result of Theorem~\ref{HUM operator}.
\hfill \qedsymbol \endproof

\begin{proof}[Proof of Lemma~\ref{lemma: mesure croisee}] The key
  point in the proof is the following lemma.
\begin{lemma}[\protect{\cite[Proposition 3.1]{Ge91}}]\label{lemGe}
Assume that $u_k$ and $v_k$ are two sequences bounded in $L^2_{loc} $
that converge weakly to zero and are associated 
with defect measures $\mu$ and $\nu$ respectively. Assume that $\mu
\perp \nu$, that is, $\mu$ and $\nu$ are supported on disjoint
sets. Then, for any $\psi \in \Con_c^0$, 
$$ \lim_{k\rightarrow +\infty} (\psi u_k, v_k)_{L^2} =0 . $$
\end{lemma} 
To apply this result, we just need to exchange the rough cutoff  $\bold{1}_{(0,T)\times \omega}$ for a smooth cutoff $\psi(t,x)$. 
  First, note that one has
  \begin{align*}
    \inp{\bold{1}_{(0,T)\times \omega}\ f^k_{\kappa, g}}
    {\d_t \ty^k}_{L^2(\L, \k \mu_{\tg} dt)}
    \ \ \mathop{\sim}_{k \to +\infty} \ \
    \inp{\bold{1}_{(0,T)\times \omega}\ f^k_{\kappa, g}}
    {\d_t \ty^k}_{L^2(\L, \k \mu_{g} dt)}.
  \end{align*}
  We may thus simply consider the $L^2$-norm and inner product
  associated with $\k \mu_{g} dt$.
  
  Second, let $\delta>0$. Since $(f^k_{\kappa, g})_k$ and $(\ty^k)_k$ are both
  bounded  in $\Con^0((0,T), L^2(\M))$ uniformly with respect to $k$,
  there exists $0< T_1 < T_2 < T$ and $\O \Subset \omega$ such that
    \begin{align*}
     \iint_{K}| f^k_{\kappa, g}|
      |\d_t\ty_k| \k\mu_{g}dt\leq \delta.
    \end{align*}
    with $K = \big( (0,T)\times \omega \big) \setminus \big(  (T_1,
    T_2)\times \O \big)$. 
    Let $\psi \in \Cinfc((0,T)\times \omega)$ such that $0\leq \psi\leq 1$ and
    equal to 1 in a \nhd of $[T_1,T_2]\times \O$. One thus has
    \begin{align*}
      \big| \inp{\bold{1}_{(0,T)\times \omega}\ f^k_{\kappa, g}}
      {\d_t \ty^k}_{L^2(\L)} \big|
      &\leq
      \big| \inp{\psi  f^k_{\kappa, g}}{\d_t \ty^k}_{L^2(\L)} \big|
      +
      \big| \inp{(\bold{1}_{(0,T)\times \omega} - \psi) f^k_{\kappa, g}}
      {\d_t \ty^k}_{L^2(\L)} \big|\\
      &\leq
      \big|  \inp{\psi  f^k_{\kappa, g}}{\d_t \ty^k}_{L^2(\L)} \big|
      +
      \delta. 
    \end{align*}
With \eqref{eq: support t mu - pertub result} and Lemma~\ref{lemGe}, one finds
\begin{align}
  \label{eq: mesure croisee}
  \inp{\psi  f^k_{\kappa, g}}
  {\d_t \ty^k}_{L^2(\L)}
  \ \ \mathop{\to}_{k \to +\infty} \ \ 0,
\end{align}
and the conclusion of the lemma follows.
\end{proof}

\subsection{Proof of Theorem ~\ref{Control of smooth data}}
\label{sec: proof Control of smooth data}

We consider first the case $\alpha =1$. As proven  in \cite{DL:09} one
has $f^{y^0,y^1}_{\k, g} \in
\Con^0([0,T], H^1(\M))$ and the estimate
\begin{align*}
  \Norm{f^{y^0,y^1}_{\k, g}}{L^\infty(0,T; H^1(\M))} \lesssim
  \Norm{(y^0,y^1)}{H^2(\M)\oplus H^1(\M)}.
\end{align*}
With this regularity of the source term in the right-hand-side of the
wave equations in \eqref{eq: stab-HUM}, one finds 
$y, \ty\in \Con^0([0,T], H^2(\M))$.
Computing the difference in \eqref{eq: stab-HUM} one writes
\begin{equation}
P_{\k, g}(y-\ty) =   (A_{\k, g} - A_{\tk, \tg}) \ty .
\end{equation}
A hyperbolic energy estimate yields
\begin{align*}
  \E_{\k, g}(y-\ty)(T)^{1/2}
  &\lesssim 
\Norm{(A_{\k, g} - A_{\tk, \tg})\ty}{L^\infty(0,T;L^2(\M))}
\lesssim \Norm{ (\k, g) -  (\tk, \tg) }{\X^1}
\Norm{\ty}{L^\infty(0,T;H^2(\M))}\\
&\lesssim \Norm{ (\k, g) -  (\tk, \tg) }{\X^1}
\Norm{f^{y^0,y^1}_{\k, g}}{L^\infty(0,T; H^1(\M))}\\
  &\lesssim
  \Norm{ (\k, g) -  (\tk, \tg) }{\X^1}
  \Norm{(y^0,y^1)}{H^2(\M)\oplus H^1(\M)}.
\end{align*}
In the case  $\alpha = 0$, one writes
\begin{align*}
  \E_{\k, g}(y-\ty)(T)^{1/2}
  &\lesssim \E_{\k,g}(y)(T)^{1/2}+\E_{\k, g}(\ty)(T)^{1/2}
  \lesssim \E_{\k,g}(y)(T)^{1/2}+\E_{\tk, \tg}(\ty)(T)^{1/2}\\
  &\lesssim
   \Norm{(y^0,y^1)}{H^1(\M)\oplus L^2(\M)}.
  \end{align*}
Finally, the result follows from interpolation between the two cases
$\alpha = 0$ 	and $\alpha = 1$.\hfill \qedsymbol \endproof


\begin{thebibliography}{10}

\bibitem{AmCr14}
L.~Ambrosio and G.~Crippa. 
\newblock Continuity equations and {ODE} flows with non-smooth velocity
\newblock {\em Proc. R. Soc. Edinb., Sect. A, Math.}, 144(6):
1191--1244, 2014.


\bibitem{BLR:87}
C.~Bardos, G.~Lebeau, and J.~Rauch.
\newblock Un exemple d'utilisation des notions de propagation pour le
  contr\^ole et la stabilisation de probl\`emes hyperboliques.
\newblock {\em Rend. Sem. Mat. Univ. Politec. Torino}, (Special Issue):11--31
  (1989), 1988.
\newblock Nonlinear hyperbolic equations in applied sciences.

\bibitem{BLR:92}
C.~Bardos, G.~Lebeau, and J.~Rauch.
\newblock Sharp sufficient conditions for the observation, control and
  stabilization of waves from the boundary.
\newblock {\em SIAM J. Control and Optim.}, 30(5):1024--1065, 1992.

\bibitem{Burq}
N.~Burq.
\newblock Contr\^olabilit\'e exacte des ondes dans des ouverts peu r\'eguliers. 
\newblock {\em Asymptotic Analysis}, 14: 157--191, 1997.

\bibitem{Bu97} N.  Burq,
\newblock Mesures semi-classiques et mesures de d\'efaut
\newblock {\em S\'eminaire Bourbaki, Ast\'erisque} Vol. 1996/97.  No. 245, Exp. No. 826, 4, 167--195, 1997.

\bibitem{B-D-LR1}
N.~Burq, B. ~Dehman and J.~Le Rousseau.
\newblock Measure and continuous vector field at a boundary {I}:
  propagation equation and wave observability, preprint 2024, arXiv 2407.02255


\bibitem{B-D-LR2}
N.~Burq, B.~Dehman and J.~Le Rousseau.
\newblock Measure and continuous vector field at a boundary {II}:
  geodesics and support propagation, preprint 2024, arXiv 2407.02259


\bibitem{BurqGerard}
N.~Burq and P.~G{\'e}rard.
\newblock Condition n\'ecessaire et suffisante pour la contr\^olabilit\'e
  exacte des ondes.
\newblock {\em C. R. Acad. Sci. Paris S\'er. I Math.}, 325(7):749--752, 1997.

\bibitem{CastroZuazua02}
C.~Castro and E.~Zuazua.
\newblock Concentration and lack of observability of waves in highly
  heterogeneous media.
\newblock {\em Arch. Ration. Mech. Anal.}, 164(1):39--72, 2002.

\bibitem{CZ07}
C.~Castro and E.~Zuazua.
\newblock Addendum to: ``{C}oncentration and lack of observability of waves in
  highly heterogeneous media'', [{A}rch. {R}ation. {M}ech. {A}nal. {\bf 164}
  (2002), no. 1, 39--72],
\newblock {\em Arch. Ration. Mech. Anal.}, 185(3):365--377, 2007.

\bibitem{Coifman-Meyer}
R.R. Coifman and Y. Meyer.
\newblock Au del\`a des op\'erateurs pseudo-diff\'erentiels.
\newblock {\em Ast\'erisque}, 57, 2013.

\bibitem{Colombini-Del Santo}
F.~Colombini, D.~Del~Santo.
\newblock A note on hyperbolic operators with log- Zygmund
coefficients.
\newblock {\em J. Math. Sci. Univ. Tokyo}, 16:95--111, 2009.



\bibitem{Colombini-DelSanto-Fanelli-Metivier}
F.~Colombini, D.~Del~Santo, F.~Fanelli, and G.~M{\'e}tivier.
\newblock Time-dependent loss of derivatives for hyperbolic operators with non
  regular coefficients.
\newblock {\em Comm. Partial Differential Equations}, 38(10):1791--1817, 2013.


\bibitem{Dehman-Ervedoza}
B.~Dehman and S.~Ervedoza.
\newblock Observability estimates for the wave equation with rough coefficients.
\newblock {\em C. R. Acad. Sci. Paris}, Ser. I 355: 499--514, 2017.


\bibitem{DL:09}
B.~Dehman and G.~Lebeau.
\newblock Analysis of the {HUM} control operator and exact controllability for
  semilinear waves in uniform time.
\newblock {\em SIAM J. Control Optim.}, 48(2):521--550, 2009.




\bibitem{Fan-Zua}
F.~Fanelli and E.~Zuazua.
\newblock Weak observability estimates for 1-{D} wave equations with rough
  coefficients.
\newblock {\em Ann. Inst. H. Poincar\'e Anal. Non Lin\'eaire}, 32(2):245--277,
  2015.
  \bibitem{Ge91}  P. G\'erard,
  \newblock  Microlocal defect measures.
  \newblock{ \em  Comm. Partial Differential Equations} 16, no. 11, 1761--1794, 1991.
  
  
\bibitem{Hormander-III}
L.~H{\"o}rmander.
\newblock {\em The analysis of linear partial differential operators. {III}},
  volume 274 of {\em Grundlehren der Mathematischen Wissenschaften}.
\newblock Springer-Verlag, Berlin, 1985.
\newblock Pseudodifferential operators.


\bibitem{Leb} 
 G. Lebeau
 \newblock {\em Control for hyperbolic equations}
Journ\'ees \'equations aux d\'eriv\'ees partielles, (1992)  1--24
 
\bibitem{Lions}
J.-L. Lions.
\newblock {\em Contr\^olabilit\'e exacte, Perturbations et Stabilisation  de
  Syst\`emes Distribu\'es. Tome 1. Contr\^olabilit\'e exacte}, volume RMA 8.
\newblock Masson, 1988.



\bibitem{MeSj78}
R.B. Melrose and J.~Sj{\"o}strand.
\newblock Singularities of boundary value problems {I}.
\newblock {\em Communications in Pure and Applied Mathematics}, 35, 1982.

\bibitem{Rauch-Taylor}
J.Rauch and M. Taylor
\newblock Exponential decay of solutions to hyperbolic equations in bounded domains.  24, 79--86,1974.
\newblock {\em Indiana Univ. Math. J.}, 24, 1--68, 1988.

\bibitem{Tartar90}
L.Tartar
\newblock $H$-measures, a new approach for studying homogenisation, oscillations and concentration effects in partial differential equations.
\newblock {\em Proc. Roy. Soc. Edinburgh Sect. A}, 115 (3-4), 193--230, 1990.

\bibitem{Zhu}
H. Zhu
\newblock  Stabilization of damped waves on spheres and Zoll surfaces of revolution..
\newblock {\em ESAIM Control Optim. Calc. Var.} 24  no. 4, 1759--1788, 2018.

\end{thebibliography}
\end{document}